\documentclass[10pt,a4paper]{article}
\usepackage[latin1]{inputenc}
\usepackage{amsmath, amssymb, mathrsfs, amsthm, amsfonts}
\usepackage{latexsym}
\usepackage[pdftex]{hyperref}

\usepackage{a4wide}

\author{Hery Randriamaro\footnote{This research was supported by DAAD.}  \\  \footnotesize{Fachbereich Mathematik und Informatik}\\  \footnotesize{Philipps-Universit\"at Marburg}\\  \footnotesize{D-35032 Marburg} \\ \small{\texttt{herand@mathematik.uni-marburg.de}}}

\title{Diagonalization of the Matrices of the Multinomial Descent and Multinomial Inversion Statistics on the Symmetric Group}

\begin{document}

\maketitle

\begin{abstract}

\noindent In the work of Varchenko, Zagier, Thibon, and Reiner, Saliola, Welker, linear algebraic properties of the multiplication map on the group algebra of the group algebra element are studied, which is the sum over all permutations weighted by \begin{math}q^{\mathtt{inv}}\end{math}, \begin{math}q^{\mathtt{maj}}\end{math}, \begin{math}\mathtt{inv}\end{math}. Here \begin{math}q\end{math} is a variable, and \begin{math}\mathtt{inv}\end{math} and \begin{math}\mathtt{maj}\end{math} are the classical statistics inversion and major index. We define a multinomial descent statistic \begin{math}\mathtt{des_X}\end{math} and a multinomial inversion statistic \begin{math}\mathtt{inv_X}\end{math}. These new defined statistics are the multinomial expressions of the classical statistics descent \begin{math}\mathtt{des}\end{math} and inversion. We determine the spectrum and the multiplicity of each element of the spectrum of the analogously defined multiplication map on the group algebra for both \begin{math}\mathtt{des_X}\end{math} and \begin{math}\mathtt{inv_X}\end{math}. As corollaries we deduce the spectrum and the multiplicity of each element of the spectrum of the defined multiplication map on the group algebra for the statistics \begin{math}\mathtt{des}\end{math}, \begin{math}\mathtt{maj}\end{math} and \begin{math}\mathtt{inv}\end{math}.
\end{abstract}

\section{Introduction}
\label{sec:introduction}

\noindent As usual we write \begin{math}\mathcal{S}_n\end{math} for the set of permutations of \begin{math}[n]:=\{1,2, \dots , n\}\end{math}, \begin{math}n \geq 1\end{math}. We consider for a permutation statistic \begin{math}s:\mathcal{S}_n \rightarrow \mathbb{R}[X_1, \dots, X_k]\end{math} the element \begin{math}\mathsf{S}:=\sum_{\sigma \in \mathcal{S}_n}s(\sigma)\sigma\end{math} of the group algebra \begin{math}\mathbb{R}[X_1, \dots, X_k][\mathcal{S}_n]\end{math}. The element \begin{math}\mathsf{S}\end{math} operates on \begin{math}\mathbb{R}[X_1, \dots, X_k][\mathcal{S}_n]\end{math} by left-multiplication as a linear mapping. The matrix of this linear mapping relativ to the standard basis \begin{math}(\sigma)_{\sigma \in \mathcal{S}_n}\end{math} is the matrix \begin{math}\big(s(\sigma \tau^{-1})\big)_{\sigma, \tau \in \mathcal{S}_n}\end{math}.\\
\noindent We recall that for \begin{math}\sigma \in \mathcal{S}_n\end{math}:
\begin{itemize}
\item its descents set is \begin{math}\mathtt{DES}(\sigma):=\{k \in [n-1]\ |\ \sigma(k)>\sigma(k+1)\}\end{math}.
\item its inversions set is \begin{math}\mathtt{INV}(\sigma):=\{(i,j)\ |\ i<j,\,\sigma(i)>\sigma(j)\}\end{math}.
\end{itemize}
We explore the matrix for the statistics \begin{math}\mathtt{des_X}: \mathcal{S}_n \rightarrow \mathbb{R}[X_1, \dots, X_{n-1}]\end{math}: \begin{displaymath}\mathtt{des_X}(\sigma):= \sum_{i \in \mathtt{DES}(\sigma)} X_i\end{displaymath}
and the statistics \begin{math}\mathtt{inv_X}: \mathcal{S}_n \rightarrow \mathbb{R}[X_{1,2}, \dots, X_{n-1,n}]\end{math}: \begin{displaymath}\mathtt{inv_X}(\sigma):= \sum_{(i,j) \in \mathtt{INV}(\sigma)} X_{i,j}.\end{displaymath}
We remark that \begin{math}\sum_{\sigma \in \mathcal{S}_n}\mathtt{des_X}(\sigma)\sigma\end{math} is an element of the descent algebra of \begin{math}\mathbb{R}[X_1, \dots, X_{n-1}][\mathcal{S}_n]\end{math}.

\newtheorem{DMatrix}{Definition} 
\begin{DMatrix}
Let $n \geq 1$. We choose a fixed linear order of \begin{math}\mathcal{S}_n\end{math} and build the following matrices whose rows and columns are indexed with the elements of \begin{math}\mathcal{S}_n\end{math} with respect to this order:
\begin{displaymath}\mathfrak{D_n}:=\big(\mathtt{des_X}(\pi \tau^{-1})\big)_{\pi, \tau \in \mathcal{S}_n},\end{displaymath}
\begin{displaymath}\mathfrak{I_n}:=\big(\mathtt{inv_X}(\pi \tau^{-1})\big)_{\pi, \tau \in \mathcal{S}_n}.\end{displaymath}
\end{DMatrix}

\noindent For a square matrix \begin{math}\mathrm{A}\end{math} we write \begin{math}Sp(\mathrm{A})\end{math} for its spectrum and \begin{math}V_{\mathrm{A}}(a)\end{math} for the multiplicity of the eigenvalue \begin{math}a \in Sp(\mathrm{A})\end{math}.\\
We want to determine \begin{math}Sp(\mathfrak{D_n})\end{math} resp. \begin{math}Sp(\mathfrak{I_n})\end{math} and the multiplicity of the elements of \begin{math}Sp(\mathfrak{D_n})\end{math} resp. \begin{math}Sp(\mathfrak{I_n})\end{math}.\\
By linear algebraic calculating we get:
\begin{enumerate}
\item[\begin{math}(n=1)\end{math}] \begin{math}Sp(\mathfrak{D_1})=\{0\}\end{math} and \begin{math}V_{\mathfrak{D_1}}(0)=1\end{math}.
\item[\begin{math}(n=2)\end{math}] \begin{math}Sp(\mathfrak{D_2})=\{X_1, -X_1\}\end{math} and \begin{math}V_{\mathfrak{D_2}}(X_1)=1\end{math}, \begin{math}V_{\mathfrak{D_2}}(-X_1)=1\end{math}.
\end{enumerate}
For \begin{math}n \geq 3\end{math}, we prove the following theorem:

\newtheorem{Gr1}{Theorem}
\begin{Gr1}
Let \begin{math}n \geq 3\end{math}. Then \begin{math}\mathfrak{D_n}\end{math} is diagonalizable and:
$$Sp(\mathfrak{D_n})=\{\frac{n!}{2}\sum_{k=1}^{n-1}X_k,\ -(n-2)!\sum_{k=1}^{n-1}X_k,\ 0\}$$
with:
\begin{itemize}
\item \begin{math}V_{\mathfrak{D_n}}\big(\frac{n!}{2}\sum_{k=1}^{n-1}X_k\big)=1\end{math},
\item \begin{math}V_{\mathfrak{D_n}}\big(-(n-2)!\sum_{k=1}^{n-1}X_k\big)=\binom{n}{2}\end{math},
\item \begin{math}V_{\mathfrak{D_n}}(0)=n! -\binom{n}{2} -1\end{math}.
\end{itemize}  
\end{Gr1}

\noindent As a direct application of the Theorem 1, we have the two following corollaries. We recall that: 
\begin{displaymath}\mathtt{des}(\sigma):= \#\mathtt{DES}(\sigma)\ \text{is the number of the descents of}\ \sigma,\end{displaymath} 
\begin{displaymath}\mathtt{maj}(\sigma):=\sum_{i \in \mathtt{DES}(\sigma)}i\ \text{is the major index of}\ \sigma.\end{displaymath}

\newtheorem{GrKoro2}{Corollary}
\begin{GrKoro2}
Let \begin{math}n \geq 1\end{math} and \begin{math}\mathsf{D_n}:=\big(\mathtt{des}(\pi \tau^{-1})\big)_{\pi, \tau \in \mathcal{S}_n}\end{math}. Then \begin{math}\mathsf{D_n}\end{math} is diagonalizable and we have:
\begin{enumerate}
\item[\begin{math}(n=1)\end{math}] \begin{math}Sp(\mathsf{D_1})=\{0\}\end{math} and \begin{math}V_{\mathsf{D_1}}(0)=1\end{math}.
\item[\begin{math}(n=2)\end{math}] \begin{math}Sp(\mathsf{D_2})=\{1, -1\}\end{math} and \begin{math}V_{\mathsf{D_2}}(1)=1\end{math}, \begin{math}V_{\mathsf{D_2}}(-1)=1\end{math}.
\item[\begin{math}(n \geq 3)\end{math}] \begin{math}Sp(\mathsf{D_n})=\{\binom{n}{2}(n-1)!,\ 0,\ -(n-1)!\}\end{math} and
\begin{itemize}
\item \begin{math}V_{\mathsf{D_n}}(\binom{n}{2}(n-1)!)=1\end{math},
\item \begin{math}V_{\mathsf{D_n}}\big(-(n-1)!\big)=\binom{n}{2}\end{math},
\item \begin{math}V_{\mathsf{D_n}}(0)=n! -\binom{n}{2} -1\end{math}.
\end{itemize}  
\end{enumerate}
\end{GrKoro2}

\begin{proof} Set \begin{math}X_i=1\end{math} in Theorem 1.  \end{proof}

\noindent The following corollary can be seen as a non-exponential version of a result by Thibon (\cite{K}, Theorem 56) who studied the matrix with entries \begin{math}q^{\mathtt{maj}(\sigma \tau^{-1})}\end{math} with deep results from the theory of noncommutative symmetric functions (\cite{G 1995} and \cite{K 1997}). It would be interesting to understand the deeper connection between the two results.

\newtheorem{MGrKoro2}[GrKoro2]{Corollary}
\begin{MGrKoro2}
Let \begin{math}n \geq 1\end{math} and \begin{math}\mathsf{M_n}:=\big(\mathtt{maj}(\pi \tau^{-1})\big)_{\pi, \tau \in \mathcal{S}_n}\end{math}. Then \begin{math}\mathsf{M_n}\end{math} diagonalizable and we have:
\begin{enumerate}
\item[\begin{math}(n=1)\end{math}] \begin{math}Sp(\mathsf{M_1})=\{0\}\end{math} and \begin{math}V_{\mathsf{M_1}}(0)=1\end{math}.
\item[\begin{math}(n=2)\end{math}] \begin{math}Sp(\mathsf{M_2})=\{1, -1\}\end{math} and \begin{math}V_{\mathsf{M_2}}(1)=1\end{math}, \begin{math}V_{\mathsf{M_2}}(-1)=1\end{math}.
\item[\begin{math}(n \geq 3)\end{math}] \begin{math}Sp(\mathsf{M_n})=\{\binom{n}{2}\frac{n!}{2},\ 0,\ -\frac{n!}{2}\}\end{math} and
\begin{itemize}
\item \begin{math}V_{\mathsf{M_n}}(\binom{n}{2}\frac{n!}{2})=1\end{math},
\item \begin{math}V_{\mathsf{M_n}}(-\frac{n!}{2})=\binom{n}{2}\end{math},
\item \begin{math}V_{\mathsf{M_n}}(0)=n! -\binom{n}{2} -1\end{math}.
\end{itemize}  
\end{enumerate}
\end{MGrKoro2}

\begin{proof} Set \begin{math}X_i=i\end{math} in Theorem 1.  \end{proof}

\bigskip

\noindent By linear algebraic calculating we get:
\begin{enumerate}
\item[\begin{math}(n=1)\end{math}] \begin{math}Sp(\mathfrak{I_1})=\{0\}\end{math} and \begin{math}V_{\mathfrak{I_1}}(0)=1\end{math}.
\item[\begin{math}(n=2)\end{math}] \begin{math}Sp(\mathfrak{I_2})=\{X_{1,2}, -X_{1,2}\}\end{math} and \begin{math}V_{\mathfrak{I_2}}(X_{1,2})=1\end{math}, \begin{math}V_{\mathfrak{I_2}}(-X_{1,2})=1\end{math}.
\item[\begin{math}(n=3)\end{math}] \begin{math}Sp(\mathfrak{I_3})=\{3X_{1,2}+3X_{1,3}+3X_{2,3},\ -X_{1,2}-2X_{1,3}-X_{2,3},\ -X_{1,2}+X_{1,3}-X_{2,3},\ 0 \}\end{math} and
\begin{itemize}
\item \begin{math}V_{\mathfrak{I_3}}(3X_{1,2}+3X_{1,3}+3X_{2,3})=1\end{math},
\item \begin{math}V_{\mathfrak{I_3}}(-X_{1,2}-2X_{1,3}-X_{2,3})=2\end{math},
\item \begin{math}V_{\mathfrak{I_3}}(-X_{1,2}+X_{1,3}-X_{2,3})=1\end{math},
\item \begin{math}V_{\mathfrak{I_3}}(0)=2\end{math}.
\end{itemize}  
\end{enumerate}
For \begin{math}n \geq 4\end{math}, we prove the following theorem:

\newtheorem{Gr2}[Gr1]{Theorem}
\begin{Gr2}
Let \begin{math}n \geq 4\end{math}. Then \begin{math}\mathfrak{I_n}\end{math} is diagonalizable and:
\begin{displaymath}Sp(\mathfrak{I_n})=\{\frac{n!}{2}\sum_{\{(i,j) \in [n]^2\ |\ i<j\}}X_{i,j},\ -(n-2)!\sum_{\{(i,j) \in [n]^2\ |\ i<j\}}(j-i)X_{i,j},\end{displaymath} \begin{displaymath}-(n-3)!\sum_{\{(i,j) \in [n]^2\ |\ i<j\}} \Big(n-2(j-i)\Big) X_{i,j},\ 0\}\end{displaymath} 
with
\begin{itemize}
\item \begin{math}V_{\mathfrak{I_n}}(\frac{n!}{2}\sum_{\{(i,j) \in [n]^2\ |\ i<j\}}X_{i,j})=1\end{math},
\item \begin{math}V_{\mathfrak{I_n}}\big(-(n-2)!\sum_{\{(i,j) \in [n]^2\ |\ i<j\}}(j-i)X_{i,j}\big)=n-1\end{math},
\item \begin{math}V_{\mathfrak{I_n}}\big(-(n-3)!\sum_{\{(i,j) \in [n]^2\ |\ i<j\}} \Big(n-2(j-i)\Big) X_{i,j}\big)=\binom{n-1}{2}\end{math},
\item \begin{math}V_{\mathfrak{I_n}}(0)=n! -\binom{n}{2} -n\end{math}.
\end{itemize}  
\end{Gr2}

\noindent We note that an exponential version of \begin{math}\mathfrak{I_n}\end{math} has been studied by Varchenko \cite{V 1993}. He was able to give a beautiful formula for its determinant. This exponential version has attracted considerable interest in various areas of mathematics (e.g.\cite{DH}, \cite{Z 1992}). As a direct application of the Theorem 2, we have the following corollary. This result was obtained in a recent independent work of Renteln (\cite{R 2011}, Section 4.8). The integrality assertion of this corollary is a very special case of a theorem of Reiner, Saliola, Welker (\cite{RSW}, Theorem 1.4) which also inspired our investigation. We recall that: 
\begin{displaymath}\mathtt{inv}(\sigma):= \#\mathtt{INV}(\sigma)\ \text{is the number of inversions of}\ \sigma.\end{displaymath} 

\newtheorem{GrKoro3}[GrKoro2]{Corollary}
\begin{GrKoro3}
Let \begin{math}n \geq 1\end{math} and \begin{math}\mathsf{I_n}:=\big(\mathtt{inv}(\pi \tau^{-1})\big)_{\pi, \tau \in \mathcal{S}_n}\end{math}. Then \begin{math}\mathsf{I_n}\end{math} is diagonalizable and we have:
\begin{enumerate}
\item[\begin{math}(n=1)\end{math}] \begin{math}Sp(\mathsf{I_1})=\{0\}\end{math} and \begin{math}V_{\mathsf{I_1}}(0)=1\end{math}.
\item[\begin{math}(n=2)\end{math}] \begin{math}Sp(\mathsf{I_2})=\{1, -1\}\end{math} and \begin{math}V_{\mathsf{I_2}}(1)=1\end{math}, \begin{math}V_{\mathsf{I_2}}(-1)=1\end{math}.
\item[\begin{math}(n=3)\end{math}] \begin{math}Sp(\mathfrak{I_3})=\{9,\ -4,\ -1,\ 0 \}\end{math} and
\begin{itemize}
\item \begin{math}V_{\mathfrak{I_3}}(9)=1\end{math},
\item \begin{math}V_{\mathfrak{I_3}}(-4)=2\end{math},
\item \begin{math}V_{\mathfrak{I_3}}(-1)=1\end{math},
\item \begin{math}V_{\mathfrak{I_3}}(0)=2\end{math}.
\end{itemize}  
\item[\begin{math}(n \geq 4)\end{math}] \begin{displaymath}Sp(\mathsf{I_n})=\{\frac{n!}{2}\binom{n}{2},\ -\frac{(n+1)!}{6},\ -\frac{n!}{6},\ 0\}\end{displaymath} and
\begin{itemize}
\item \begin{math}V_{\mathsf{I_n}}(\frac{n!}{2}\binom{n}{2})=1\end{math},
\item \begin{math}V_{\mathsf{I_n}}(\frac{(n+1)!}{6})=n-1\end{math},
\item \begin{math}V_{\mathsf{I_n}}(\frac{n!}{6})=\binom{n}{2}\end{math},
\item \begin{math}V_{\mathsf{I_n}}(0)=n! -\binom{n}{2} -n\end{math}.
\end{itemize}  
\end{enumerate}
\end{GrKoro3}

\begin{proof} Set \begin{math}X_{i,j}=1\end{math} in Theorem 2.  \end{proof}

\section{Multinomial Version of the Theorem of Perron-Frobenius}
\label{sec:perron}

\noindent For a polynomial \begin{math}P\end{math} and a monomial \begin{math}M\end{math} in \begin{math}\mathbb{R}[X_1, \dots, X_k]\end{math}, we write \begin{math}[M]P\end{math} for the coefficient of \begin{math}M\end{math} in \begin{math}P\end{math}. For a square matrix \begin{math}\mathsf{A}\end{math} we write \begin{math}E_{\mathsf{A}}(a)\end{math} for the eigenspace of the eigenvalue \begin{math}a \in Sp(\mathsf{A})\end{math} and for a vector \begin{math}\mathsf{v}\end{math} we write \begin{math}<\mathsf{v}>\end{math} for the subspace generated by \begin{math}\mathsf{v}\end{math}.

\newtheorem{XMNFijI}{Proposition}[section]
\begin{XMNFijI}
Let \begin{math}n \geq 2\end{math} and \begin{math}\mathfrak{P_n}= \big(P_{i,j}\big)_{i,j \in [n]}\end{math} be a matrix of polynomials \begin{math}P_{i,j} \in \mathbb{R}[X_1, \dots, X_k]\end{math} such that:
\begin{itemize}
\item[(a)] \begin{math}P_{i,j} \neq 0\end{math} and \begin{math}[X_1^{i_1} \dots X_k^{i_k}]P_{i,j} \geq 0\end{math},
\item[(b)] there is a polynomial \begin{math}P_n \in \mathbb{R}[X_1, \dots, X_k]\end{math} such that, for any \begin{math}i', i'' \in [n]\end{math},  \begin{displaymath}\sum_{j=1}^n P_{i', j}= \sum_{j=1}^n P_{i'', j}= P_n.\end{displaymath}
Then \begin{math}P_n \in Sp(\mathfrak{P_n})\end{math} and \begin{math}E_{\mathfrak{P_n}}(P_n)=<\left( \begin{array}{c} 1 \\ \vdots \\ 1 \end{array} \right)>\end{math}.
\end{itemize} 
\end{XMNFijI}

\begin{proof} We prove the assertion by induction on \begin{math}n\end{math}. With simple calculations, we get the result for \begin{math}n=2\end{math}.\\
Now we assume that the assertion is proven for \begin{math}n-1 \geq 2\end{math}.\\
Let us consider the following system of n equations:
\begin{equation*}\left. \begin{array}{ccccccccccc}
P_{1,1}x_1 & + & P_{1,2}x_2 & + & \dots & + & P_{1,n-1}x_{n-1} & + & P_{1,n}x_n &  = & P_n x_1\\
P_{2,1}x_1 & + & P_{2,2}x_2 & + & \dots & + & P_{2,n-1}x_{n-1} & + & P_{2,n}x_n & = & P_n x_2\\
\vdots && \vdots &&&& \vdots && \vdots && \vdots \\
P_{n-1,1}x_1 & + & P_{n-1,2}x_2 & + & \dots & + & P_{n-1,n-1}x_{n-1} & + & P_{n-1,n}x_n & = & P_n x_{n-1}\\
P_{n,1}x_1 & + & P_{n,2}x_2 & + & \dots & + & P_{n,n-1}x_{n-1} & + & P_{n,n}x_n & = & P_n x_n
\end{array} \right. \end{equation*}
On the last row, we have \begin{math}P_{n,1}x_1+ P_{n,2}x_2+ \dots + P_{n,n-1}x_{n-1}= (P_n- P_{n,n}) x_n\end{math}. By multiplying the n-1 first rows with \begin{math}(P_n- P_{n,n})\end{math}, and after by replacing \begin{math}(P_n- P_{n,n}) x_n\end{math} with \begin{math}P_{n,1}x_1+ P_{n,2}x_2+ \dots + P_{n,n-1}x_{n-1}\end{math} in the n-1 first rows, we obtain the following system of n-1 equations:
\begin{equation*}\left. \begin{array}{ccccccccc}
P_{1,1}(P_n- P_{n,n})x_1 & + & \dots & + & P_{1,n-1}(P_n- P_{n,n})x_{n-1} & + & P_{1,n}Q & = & P_{n-1} x_1\\
P_{2,1}(P_n- P_{n,n})x_1 & + & \dots & + & P_{2,n-1}(P_n- P_{n,n})x_{n-1} & + & P_{2,n}Q & = & P_{n-1} x_2\\
\vdots  &&&& \vdots && \vdots && \vdots \\
P_{n-1,1}(P_n- P_{n,n})x_1 & + & \dots & + & P_{n-1,n-1}(P_n- P_{n,n})x_{n-1} & + & P_{n-1,n}Q & = & P_{n-1} x_{n-1}
\end{array} \right. \end{equation*}
where \begin{math}P_{n-1}= P_n(P_n- P_{n,n})\end{math} and \begin{math}Q= P_{n,1}x_1+ P_{n,2}x_2+ \dots + P_{n,n-1}x_{n-1}\end{math}.\\
We have, for \begin{math}i \in [n-1]\end{math}:
\begin{itemize}
\item 
\begin{displaymath}\sum_{j=1}^{n-1}[x_j]\big(P_{i,1}(P_n- P_{n,n})x_1+ \dots + P_{i,n-1}(P_n- P_{n,n})x_{n-1} + P_{i,n}Q\big)\end{displaymath}
\begin{displaymath}= \sum_{j=1}^{n-1}P_{i,j}(P_n- P_{n,n}) + \sum_{j=1}^{n-1}P_{i,n}P_{n,j}\end{displaymath}
\begin{displaymath}= (P_n- P_{i,n})(P_n- P_{n,n}) + P_{i,n}(P_n- P_{n,n})\end{displaymath}
\begin{displaymath}= P_{n-1}.\end{displaymath}
\item Let \begin{math}j \in [n-1]\end{math} and set \begin{displaymath}P'_{i,j}=[x_j]\big(P_{i,1}(P_n- P_{n,n})x_1+ \dots + P_{i\,n-1}(P_n- P_{n,n})x_{n-1} + P_{i,n}Q\big).\end{displaymath} By (a) and the definition of \begin{math}\mathfrak{P_n}\end{math}, we have that \begin{math}0 \neq P_{i,j}(P_n- P_{n,n})\end{math} is a polynomial with non-negative coefficients only. Again by (a), it follows that:
\begin{itemize}
\item[\begin{math}\triangleright\end{math}] \begin{math}P'_{i,j}= P_{i,j}(P_n- P_{n,n}) + P_{i,n}P_{n,j} \neq 0\end{math},
\item[\begin{math}\triangleright\end{math}] \begin{math}(P'_{i,j}\,,\,X_1^{i_1} \dots X_k^{i_k}) \geq 0\end{math}, for \begin{math}i_1, \dots, i_k \in \mathbb{N}\end{math}.
\end{itemize}
\end{itemize}
By induction, the solution of the system (2) is \begin{math}x_1 = x_2 = \dots = x_{n-1}\end{math}. By replacing \begin{math}x_2, x_3, \dots, x_{n-1}\end{math} by \begin{math}x_1\end{math} in \begin{math}P_{n,1}x_1+ P_{n,2}x_2+ \dots + P_{n,n-1}x_{n-1}= (P_n- P_{n,n}) x_n\end{math}, we have then \begin{math}x_1 = x_n\end{math}. Hence the solution of this system (1) is \begin{math}x_1 = x_2 = \dots = x_n\end{math}.\\
Therefore, \begin{math}P_n \in Sp(\mathfrak{P_n})\end{math} and \begin{math}E_{\mathfrak{P_n}}(P_n)=<\left( \begin{array}{c} 1 \\ \vdots \\ 1 \end{array} \right)>\end{math}. 
\end{proof}

\newtheorem{echt}[XMNFijI]{Corollary}
\begin{echt}
Let \begin{math}n \geq 2\end{math} and \begin{math}\mathfrak{P_n}= \big(P_{i,j}\big)_{i,j \in [n]}\end{math} be a matrix of polynomials \begin{math}P_{i,j} \in \mathbb{R}[X_1, \dots, X_k]\end{math} such that:
\begin{itemize}
\item[(a)] \begin{math}P_{i,j} \neq 0\end{math} and \begin{math}[X_1^{i_1} \dots X_k^{i_k}]P_{i,j} \geq 0\end{math}, for \begin{math}i \neq j\end{math}, 
\item[(b)] \begin{math}P_{i,i} = 0\end{math}, for \begin{math}i \in [n]\end{math}, 
\item[(c)] there is a polynomial \begin{math}P_n \in \mathbb{R}[X_1, \dots, X_k]\end{math} such that, for any \begin{math}i', i'' \in [n]\end{math}, \begin{displaymath}\sum_{j=1}^n P_{i', j}= \sum_{j=1}^n P_{i'', j}= P_n.\end{displaymath}
Then \begin{math}P_n \in Sp(\mathfrak{P_n})\end{math} and \begin{math}E_{\mathfrak{P_n}}(P_n)=<\left( \begin{array}{c} 1 \\ \vdots \\ 1 \end{array} \right)>\end{math}.
\end{itemize} 
\end{echt}

\begin{proof} The proof is obvious for \begin{math}n=2\end{math}.\\ Let us consider \begin{math}n>2\end{math}. We have the following system of n equations:
\begin{equation*}\left. \begin{array}{ccccccccccc}
0 & + & P_{1,2}x_2 & + & \dots & + & P_{1,n-1}x_{n-1} & + & P_{1,n}x_n & = & P_n x_1\\
P_{2,1}x_1 & + & 0 & + & \dots & + & P_{2,n-1}x_{n-1} & + & P_{2,n}x_n & = & P_n x_2\\
\vdots  && \vdots &&&& \vdots && \vdots && \vdots \\
P_{n-1,1}x_1 & + & P_{n-1,2}x_2 & + & \dots & + & 0 & + & P_{n-1,n}x_n & = & P_n x_{n-1}\\
P_{n,1}x_1 & + & P_{n,2}x_2 & + & \dots & + & P_{n,n-1}x_{n-1} & + & 0 & = & P_n x_n
\end{array} \right. \end{equation*}
On the last row, we have \begin{math}P_{n,1}x_1+ P_{n,2}x_2+ \dots + P_{n,n-1}x_{n-1}= P_n x_n\end{math}. By multiplying the n-1 first rows with \begin{math}P_n\end{math}, and after by replacing \begin{math}P_n x_n\end{math} with \begin{math}P_{n,1}x_1+ P_{n,2}x_2+ \dots + P_{n,n-1}x_{n-1}\end{math} in the n-1 first rows, we obtain the following system of n-1 equations:
\begin{equation*}\left. \begin{array}{ccccccc}
P_{1,n} P_{n,1}x_1 & + & \dots & + & (P_n P_{1,n-1} + P_{1,n}P_{n,n-1})x_{n-1} & = & P_n^2 x_1\\
(P_n P_{2,1} + P_{2,n}P_{n,1})x_1 & + & \dots & + & (P_n P_{2,n-1} + P_{2,n}P_{n,n-1})x_{n-1} & = & P_n^2 x_2\\
\vdots &&&& \vdots && \vdots \\
(P_n P_{n-1,1} + P_{n-1,n}P_{n,1})x_1 & + & \dots & + & P_{n-1,n} P_{n,n-1}x_{n-1} & = & P_n^2 x_{n-1}
\end{array} \right. \end{equation*}
where 
\begin{displaymath}\sum_{j=1}^{n-1}[x_j]\big( \sum_{k=1}^{n-1} [x_k](P_n P_{i,k} + P_{i,n}P_{n,k}) \big)=  P_n \sum_{k=1}^{n-1}P_{i,k} + P_{i,n}\sum_{k=1}^{n-1}P_{n,k}= P_n \sum_{k=1}^{n-1}P_{i,k} + P_n P_{i,n}= P_n^2.\end{displaymath}
Using Proposition 2.1 on the obtained system, we get \begin{math}x_1 = x_2 = \dots = x_{n-1}\end{math} and then the desided result.
\end{proof}

\noindent Proposition 2.1 can also be applied to get more general results. As example, for a matrix of polynomials \begin{math}\mathfrak{P_n}= \big(P_{i,j}\big)_{i,j \in [n]}\end{math} such that:
\begin{itemize}
\item[(a)] \begin{math}[X_1^{i_1} \dots X_k^{i_k}]P_{i,j} \geq 0\end{math},
\item[(b)] \begin{math}P_{i,n} \neq 0\end{math} and \begin{math}P_{n,i} \neq 0\end{math}, for \begin{math}i \neq n\end{math},
\item[(c)] there is a polynomial \begin{math}P_n \in \mathbb{R}[X_1, \dots, X_k]\end{math} such that, for any \begin{math}i', i'' \in [n]\end{math}, \begin{displaymath}\sum_{j=1}^n P_{i', j}= \sum_{j=1}^n P_{i'', j}= P_n.\end{displaymath}
\end{itemize}
Then \begin{math}P_n \in Sp(\mathfrak{P_n})\end{math} and \begin{math}E_{\mathfrak{P_n}}(P_n)=<\left( \begin{array}{c} 1 \\ \vdots \\ 1 \end{array} \right)>\end{math}.

\noindent Now we apply Corollary 2.2 to the statistics \begin{math}\mathtt{des_X}\end{math} and \begin{math}\mathtt{inv_X}\end{math}.

\newtheorem{SatzEr}[XMNFijI]{Lemma}
\begin{SatzEr}
For \begin{math}n \geq 3\end{math}, we have \begin{displaymath}\sum_{\sigma \in \mathcal{S}_n} \mathtt{des_X}(\sigma) = \frac{n!}{2}\sum_{k=1}^{n-1}X_k.\end{displaymath}
\end{SatzEr}

\begin{proof} For \begin{math}k \in [n-1]\end{math}, \begin{math}[X_k]\big(\sum_{\sigma \in \mathcal{S}_n} \mathtt{des_X}(\sigma)\big)= \#\{\sigma \in \mathcal{S}_n\,|\,k \in \mathtt{DES}(\sigma)\}\end{math}. It is clear that the following mapping is bijective: 
\begin{displaymath}\gamma: \left\{ 
\begin{array}{ccc}
\{\sigma \in \mathcal{S}_n\,|\,k \in \mathtt{DES}(\sigma)\} & \rightarrow & \{\sigma \in \mathcal{S}_n\,|\,k \notin \mathtt{DES}(\sigma)\}\\
\tau= \left( \begin{array}{cccc} \dots & k & k+1 & \dots  \\ \  \dots & \tau(k) & \tau(k+1) & \dots   \end{array} \right)
& \mapsto & 
\gamma(\tau) = \left( \begin{array}{cccc}  \dots & k & k+1 & \dots   \\   \dots & \tau(k+1) & \tau(k) & \dots  \end{array} \right)
\end{array}
\right.,\end{displaymath}
that means \begin{math}\#\{\sigma \in \mathcal{S}_n\,|\,k \in \mathtt{DES}(\sigma)\}= \#\{\sigma \in \mathcal{S}_n\,|\,k \notin \mathtt{DES}(\sigma)\}\end{math}.\\
Since \begin{math}\#\{\sigma \in \mathcal{S}_n\,|\,k \in \mathtt{DES}(\sigma)\}\ +\ \#\{\sigma \in \mathcal{S}_n\,|\,k \notin \mathtt{DES}(\sigma)\}= \#\mathcal{S}_n\end{math}, then \begin{math}[X_j]\big(\sum_{\sigma \in \mathcal{S}_n} \mathtt{des_X}(\sigma)\big)= \frac{n!}{2}\end{math}, and \begin{math}\sum_{\sigma \in \mathcal{S}_n} \mathtt{des_X}(\sigma) = \frac{n!}{2}\sum_{k=1}^{n-1}X_k.\end{math} 
\end{proof}

\noindent From Corollary 2.2 and Lemma 2.3, we deduce that \begin{equation}\frac{n!}{2}\sum_{k=1}^{n-1}X_k \in Sp(\mathfrak{D_n})\ \text{and}\ E_{\mathfrak{D_n}}(\frac{n!}{2}\sum_{k=1}^{n-1}X_k)=<\left( \begin{array}{c} 1 \\ \vdots \\ 1 \end{array} \right)>.\label{Dn}\end{equation}
Besides, from Lemma 2.3, we also directly get \begin{displaymath}\sum_{\sigma \in \mathcal{S}_n} \mathtt{des}(\sigma) = \frac{(n-1)n!}{2}\ \text{and}\ \sum_{\sigma \in \mathcal{S}_n} \mathtt{maj}(\sigma) = \frac{n!}{2}\binom{n}{2}\end{displaymath}
that are differently calculated in other books (\cite{S 1997}, Example 2.2.5 resp. Corollary 4.5.9 for example).\\

\newtheorem{SatzEr2}[XMNFijI]{Lemma}
\begin{SatzEr2}
For \begin{math}n \geq 4\end{math}, we have \begin{displaymath}\sum_{\sigma \in \mathcal{S}_n} \mathtt{inv_X}(\sigma) = \frac{n!}{2}\sum_{\{(i,j) \in [n]^2\ |\ i<j\}}X_{i,j}.\end{displaymath}
\end{SatzEr2}

\begin{proof} For \begin{math}i,j \in [n]\end{math}, \begin{math}i<j\end{math}, \begin{math}[X_{i,j}]\big(\sum_{\sigma \in \mathcal{S}_n} \mathtt{inv_X}(\sigma)\big)= \#\{\sigma \in \mathcal{S}_n\,|\,(i,j) \in \mathtt{INV}(\sigma)\}\end{math}. It is clear that the following mapping is bijective: 
\begin{displaymath}\varphi: \left\{ 
\begin{array}{ccc}
\{\sigma \in \mathcal{S}_n\,|\,(i,j) \in \mathtt{INV}(\sigma)\} & \rightarrow & \{\sigma \in \mathcal{S}_n\,|\,(i,j) \notin \mathtt{INV}(\sigma)\}\\
\tau= \left( \begin{array}{ccccc}  \dots & i & \dots & j & \dots  \\ \ \dots & \tau(i) & \dots & \tau(j) & \dots  \end{array} \right)
& \mapsto & 
\varphi(\tau) = \left( \begin{array}{ccccc}  \dots & i & \dots & j & \dots  \\  \dots & \tau(j) & \dots & \tau(i) & \dots \end{array} \right)
\end{array}
\right.,\end{displaymath}
that means \begin{math}\#\{\sigma \in \mathcal{S}_n\,|\,(i,j) \in \mathtt{INV}(\sigma)\}= \#\{\sigma \in \mathcal{S}_n\,|\,(i,j) \notin \mathtt{INV}(\sigma)\}\end{math}.\\
Since \begin{math}\#\{\sigma \in \mathcal{S}_n\,|\,(i,j) \in \mathtt{INV}(\sigma)\}\ +\ \#\{\sigma \in \mathcal{S}_n\,|\,(i,j) \notin \mathtt{INV}(\sigma)\}= \#\mathcal{S}_n\end{math}, then \begin{math}[X_{i,j}]\big(\sum_{\sigma \in \mathcal{S}_n} \mathtt{inv_X}(\sigma)\big)= \frac{n!}{2}\end{math}, and \begin{math}\sum_{\sigma \in \mathcal{S}_n} \mathtt{inv_X}(\sigma) = \frac{n!}{2}\sum_{\{(i,j) \in [n]^2\ |\ i<j\}}X_{i,j}.\end{math}  
\end{proof}

\noindent From Corollary 2.2 and Lemma 2.4, we deduce that \begin{equation}\frac{n!}{2}\sum_{\{(i,j) \in [n]^2\ |\ i<j\}}X_{i,j} \in Sp(\mathfrak{I_n})\ \text{and}\ E_{\mathfrak{I_n}}(\frac{n!}{2}\sum_{\{(i,j) \in [n]^2\ |\ i<j\}}X_{i,j})=<\left( \begin{array}{c} 1 \\ \vdots \\ 1 \end{array} \right)>. \label{In}\end{equation}
Besides, from Lemma 2.4, we also directly get \begin{displaymath}\sum_{\sigma \in \mathcal{S}_n} \mathtt{inv}(\sigma) = \frac{n!}{2}\binom{n}{2}\end{displaymath}
that is differently calculated in other books (\cite{S 1997}, Corollary 1.3.10 for example).\\

\section{Minimal Polynomial of the Multinomial Descent Statistic}
\label{sec:descent}

\noindent To determine the minimal polynomial of \begin{math}\mathfrak{D_n}\end{math} we need the both following lemmas:

\noindent Let \begin{math}n \geq 3\end{math} and \begin{math}i,j \in [n]\end{math}. We write
\begin{displaymath}\mathcal{S}_n^{i,j}:= \{ \sigma \in \mathcal{S}_n\ |\ \sigma^{-1}(j)-\sigma^{-1}(i)=1 \},\end{displaymath}
\begin{displaymath}\mathcal{S}_n^{i\,-\,j}:= \{ \sigma \in \mathcal{S}_n\ |\ \sigma^{-1}(j)-\sigma^{-1}(i)>1 \}.\end{displaymath} 

\noindent Let \begin{math}n \geq 3\end{math}. We write \begin{displaymath}\mathfrak{d_n}= \sum_{k=1}^{n-1}X_k.\end{displaymath}

\newtheorem{super}[XMNFijI]{Lemma}
\begin{super}
Let \begin{math}n \geq 3\end{math} and \begin{math}i,j \in [n]\end{math}. Then
\begin{itemize}
\item if \begin{math}j>i+1\end{math}: \begin{displaymath}\sum_{\sigma \in \mathcal{S}_n^{i,j}} X_{\sigma^{-1}(i)} \mathtt{des_X}(\sigma^{-1})= \sum_{\sigma \in \mathcal{S}_n^{j,i}} X_{\sigma^{-1}(j)} \mathtt{des_X}(\sigma^{-1}).\end{displaymath}
\item if \begin{math}j=i+1\end{math}: \begin{displaymath}\sum_{\sigma \in \mathcal{S}_n^{i,(i+1)}} X_{\sigma^{-1}(i)} \mathtt{des_X}(\sigma^{-1})= \sum_{\sigma \in \mathcal{S}_n^{(i+1),i}} X_{\sigma^{-1}(i+1)} \mathtt{des_X}(\sigma^{-1})\ -X_i (n-2)!\mathfrak{d_n}.\end{displaymath}
\end{itemize}
\end{super}

\begin{proof} Let us consider the following bijective mapping: 
\begin{displaymath}\kappa_{i,j}: \left\{ 
\begin{array}{ccc}
\mathcal{S}_n & \rightarrow & \mathcal{S}_n\\
\sigma= \left( \begin{array}{ccccc}  \dots & i & \dots & j & \dots  \\ \ \dots & \sigma(i) & \dots & \sigma(j) & \dots  \end{array} \right)
& \mapsto & 
\kappa_{i,j}(\sigma) = \left( \begin{array}{ccccc}  \dots & i & \dots & j & \dots  \\  \dots & \sigma(j) & \dots & \sigma(i) & \dots \end{array} \right)
\end{array}.
\right.\end{displaymath}

\begin{itemize}
\item Let \begin{math}\sigma \in \mathcal{S}_n^{i\,j}\end{math}. Then we have the following simple facts:
\begin{itemize}
\item[\begin{math}\triangleright\end{math}] \begin{math}\sigma^{-1}(j)= \sigma^{-1}(i)+1\end{math},
\item[\begin{math}\triangleright\end{math}] \begin{math}\sigma^{-1} \in \mathcal{S}_n^{\sigma^{-1}(i)\,-\,\sigma^{-1}(j)}\end{math},
\item[\begin{math}\triangleright\end{math}] \begin{math}\kappa_{i,j}(\sigma^{-1}) \in \mathcal{S}_n^{\sigma^{-1}(j)\,-\,\sigma^{-1}(i)}\end{math},
\item[\begin{math}\triangleright\end{math}] \begin{math}\mathtt{des_X}(\sigma^{-1})= \mathtt{des_X}\big(\kappa_{i,j}(\sigma^{-1})\big)\end{math},
\item[\begin{math}\triangleright\end{math}] \begin{math}\big(\kappa_{i,j}(\sigma^{-1})\big)^{-1} \in \mathcal{S}_n^{j,i}\end{math}.
\end{itemize}
Thus
\begin{displaymath}\sum_{\sigma \in \mathcal{S}_n^{i,j}} X_{\sigma^{-1}(i)} \mathtt{des_X}(\sigma^{-1})= \sum_{\sigma \in \mathcal{S}_n^{i,j}} X_{\sigma^{-1}(i)} \mathtt{des_X}\big(\kappa_{i,j}(\sigma^{-1})\big)= \sum_{\sigma \in \mathcal{S}_n^{j,i}} X_{\sigma^{-1}(j)} \mathtt{des_X}(\sigma^{-1}). \end{displaymath}

\item Let \begin{math}\sigma \in \mathcal{S}_n^{i\,(i+1)}\end{math}. Again the following facts hold:
\begin{itemize}
\item[\begin{math}\triangleright\end{math}] \begin{math}\sigma^{-1} \in \mathcal{S}_n^{\sigma^{-1}(i),(\sigma^{-1}(i)+1)}\end{math},
\item[\begin{math}\triangleright\end{math}] \begin{math}\kappa_{i,i+1}(\sigma^{-1}) \in \mathcal{S}_n^{(\sigma^{-1}(i)+1),\sigma^{-1}(i)}\end{math},
\item[\begin{math}\triangleright\end{math}] \begin{math}\mathtt{des_X}(\sigma^{-1})= \mathtt{des_X}\big(\kappa_{i,i+1}(\sigma^{-1})\big)- X_i\end{math},
\item[\begin{math}\triangleright\end{math}] \begin{math}\big(\kappa_{i,i+1}(\sigma^{-1})\big)^{-1} \in \mathcal{S}_n^{(i+1),i}\end{math}.
\end{itemize}
Thus
\begin{displaymath}\sum_{\sigma \in \mathcal{S}_n^{i\,(i+1)}} X_{\sigma^{-1}(i)} \mathtt{des_X}(\sigma^{-1})= \sum_{\sigma \in \mathcal{S}_n^{i,(i+1)}} X_{\sigma^{-1}(i)}  \Big(\mathtt{des_X}\big(\kappa_{i,i+1}(\sigma^{-1})\big)- X_i\Big)\end{displaymath}
\begin{displaymath}= \sum_{\sigma \in \mathcal{S}_n^{i\,(i+1)}} X_{\sigma^{-1}(i)} \mathtt{des_X}\big(\kappa_{i,i+1}(\sigma^{-1})\big)- \sum_{\sigma \in \mathcal{S}_n^{i,(i+1)}} X_{\sigma^{-1}(i)} X_i= \sum_{\sigma \in \mathcal{S}_n^{(i+1),i}} X_{\sigma^{-1}(i+1)} \mathtt{des_X}(\sigma^{-1}) -X_i\sum_{k=1}^{n-1}\sum_{\substack{\sigma \in \mathcal{S}_n^{i,(i+1)} \\ \sigma(k)=i}} X_k\end{displaymath}
\begin{displaymath}= \sum_{\sigma \in \mathcal{S}_n^{(i+1),i}} X_{\sigma^{-1}(i+1)} \mathtt{des_X}(\sigma^{-1}) -X_i (n-2)!\mathfrak{d_n}.\end{displaymath}
\end{itemize}
\end{proof}

\noindent Let \begin{math}n \geq 3\end{math}. We write \begin{math}\iota\end{math} for the identity permutation of \begin{math}\mathcal{S}_n\end{math} and define:
\begin{displaymath}\left.\begin{array}{cccc}
\Pi_n: & \mathcal{S}_n & \rightarrow & \mathbb{R}[X_1, \dots, X_{n-1}]\\
   & \tau & \mapsto & \sum_{\sigma \in \mathcal{S}_n}  \mathtt{des_X}(\tau \sigma)\mathtt{des_X}(\sigma^{-1})
\end{array}\right.\end{displaymath}

\newtheorem{XMZGrSatz}[XMNFijI]{Lemma}
\begin{XMZGrSatz}
Let \begin{math}n \geq 3\end{math} and \begin{math}\tau \in \mathcal{S}_n\end{math}. Then
\begin{displaymath}\Pi_n(\tau) = \Pi_n(\iota) - \mathtt{des_X}(\tau)(n-2)!\mathfrak{d_n}.\end{displaymath}
\end{XMZGrSatz}

\begin{proof} Every permutation \begin{math}\tau \in \mathcal{S}_n\end{math} can be written as \begin{math}\tau=\tau_k \dots \tau_2 \tau_1 \iota\end{math} for some \begin{math}\tau_i \in \{(j\ j+1)\,|\,1 \leq j \leq n-1\}\end{math}. We consider the following mapping: 
\begin{displaymath}\phi: \left\{ 
\begin{array}{ccc}
\mathcal{S}_n & \rightarrow & \mathbb{R}[X_1, \dots, X_{n-1}]\\
\tau=\tau_k \dots \tau_2 \tau_1 \iota & \mapsto & \Big(\mathtt{des_X}(\tau)- \mathtt{des_X}(\tau_{k-1} \dots \tau_2 \tau_1 \iota)\Big) + \phi\Big(\tau_{k-1} \dots \tau_2 \tau_1 \iota \Big) 
\end{array}
\right.,\end{displaymath}
\begin{displaymath}\text{where}\ \phi(\iota)=0\ \text{and}\ \phi\Big((j\ j+1)\Big)=X_j.\end{displaymath}
We are going to prove that \begin{math} \Pi_n(\tau) = \Pi_n(\iota) - \phi(\tau)(n-2)!\mathfrak{d_n}\end{math}.\\
The case \begin{math}k=0\end{math} and \begin{math}\tau=\iota\end{math} is trivial.\\
For \begin{math}k=1\end{math} and \begin{math}\tau= (j\ j+1)\end{math} we have:
\begin{displaymath}\Pi_n\big((j\ j+1)\big) = \sum_{\sigma \in \mathcal{S}_n^{j\,-\,(j+1)}} \mathtt{des_X}\Big((j\ j+1)\sigma\Big) \mathtt{des_X}\Big(\sigma^{-1}\Big)+\sum_{\sigma \in \mathcal{S}_n^{(j+1)\,-\,j}} \mathtt{des_X}\Big((j\ j+1)\sigma\Big) \mathtt{des_X}\Big(\sigma^{-1}\Big)+  \end{displaymath}
 \begin{displaymath}\sum_{\sigma \in \mathcal{S}_n^{j,(j+1)}} \mathtt{des_X}\Big((j\ j+1)\sigma\Big) \mathtt{des_X}\Big(\sigma^{-1}\Big)+\sum_{\sigma \in \mathcal{S}_n^{(j+1),j}} \mathtt{des_X}\Big((j\ j+1)\sigma\Big) \mathtt{des_X}\Big(\sigma^{-1}\Big)=\end{displaymath}

\begin{displaymath}\sum_{\sigma \in \mathcal{S}_n^{j\,-\,(j+1)}} \mathtt{des_X}(\sigma) \mathtt{des_X}(\sigma^{-1})+ \sum_{\sigma \in \mathcal{S}_n^{(j+1)\,-\,j}} \mathtt{des_X}(\sigma) \mathtt{des_X}(\sigma^{-1})+ \end{displaymath}
\begin{displaymath}\sum_{\sigma \in \mathcal{S}_n^{j,(j+1)}} \Big(\mathtt{des_X}(\sigma)+ X_{\sigma^{-1}(j)}\Big)\mathtt{des_X}(\sigma^{-1})+ \sum_{\sigma \in \mathcal{S}_n^{(j+1),j}} \Big(\mathtt{des_X}(\sigma)- X_{\sigma^{-1}(j+1)}\Big)\mathtt{des_X}(\sigma^{-1})= \end{displaymath}

\begin{displaymath}\Pi_n(\iota)+ \sum_{\sigma \in \mathcal{S}_n^{j,(j+1)}} X_{\sigma^{-1}(j)} \mathtt{des_X}(\sigma^{-1})- \sum_{\sigma \in \mathcal{S}_n^{(j+1),j}} X_{\sigma^{-1}(j+1)} \mathtt{des_X}(\sigma^{-1})=   \end{displaymath}

\begin{displaymath}\Pi_n(\iota)+ \sum_{\sigma \in \mathcal{S}_n^{(j+1),j}} X_{\sigma^{-1}(j+1)}\mathtt{des_X}(\sigma^{-1})-  X_j(n-2)!\mathfrak{d_n}- \sum_{\sigma \in \mathcal{S}_n^{(j+1),j}} X_{\sigma^{-1}(j+1)} \mathtt{des_X}(\sigma^{-1})= \end{displaymath}

\begin{displaymath}\Pi_n(\iota)- \phi\Big((j\ j+1)\Big)(n-2)!\mathfrak{d_n}.\end{displaymath}

\noindent Now we assume that the assertion is proven for \begin{math}k-1 \geq 1\end{math}.\\
Let \begin{math}k \geq 2\end{math} and \begin{math}\tau=\tau_k \dots \tau_2 \tau_1 \iota =(j\ j+1)\tau'\end{math}, where \begin{math}\tau_k= (j\ j+1)\end{math} and \begin{math}\tau'=\tau_{k-1} \dots \tau_2 \tau_1 \iota\end{math}:
\begin{displaymath}\Pi_n\big((j\ j+1) \tau'\big)= \sum_{\{\sigma \in \mathcal{S}_n\,|\,\tau' \sigma \in \mathcal{S}_n^{j\,-\,(j+1)}\}} \mathtt{des_X}\Big((j\ j+1)\tau' \sigma\Big) \mathtt{des_X}\Big(\sigma^{-1}\Big)+\sum_{\{\sigma \in \mathcal{S}_n\,|\,\tau'\sigma \in \mathcal{S}_n^{(j+1)\,-\,j}\}} \mathtt{des_X}\Big((j\ j+1)\tau'\sigma\Big) \mathtt{des_X}\Big(\sigma^{-1}\Big)+ \end{displaymath}  \begin{displaymath} \sum_{\{\sigma \in \mathcal{S}_n\,|\,\tau' \sigma \in \mathcal{S}_n^{j,(j+1)} \}} \mathtt{des_X}\Big((j\ j+1)\tau'\sigma\Big) \mathtt{des_X}\Big(\sigma^{-1}\Big)+ \sum_{\{\sigma \in \mathcal{S}_n\,|\,\tau' \sigma \in \mathcal{S}_n^{(j+1),j} \}} \mathtt{des_X}\Big((j\ j+1)\tau'\sigma\Big) \mathtt{des_X}\Big(\sigma^{-1}\Big)=\end{displaymath}

\begin{displaymath}\sum_{\{\sigma \in \mathcal{S}_n\,|\,\tau'\sigma \in \mathcal{S}_n^{j\,-\,(j+1)}\}} \mathtt{des_X}(\tau' \sigma) \mathtt{des_X}(\sigma^{-1})+ \sum_{\{\sigma \in \mathcal{S}_n\,|\,\tau'\sigma \in \mathcal{S}_n^{(j+1)\,-\,j}\}} \mathtt{des_X}(\tau' \sigma) \mathtt{des_X}(\sigma^{-1})+ \end{displaymath}
\begin{displaymath}\sum_{\{\sigma \in \mathcal{S}_n\,|\,\tau' \sigma \in \mathcal{S}_n^{j,(j+1)} \}} \Big(\mathtt{des_X}(\tau'\sigma)+ X_{\sigma^{-1}\tau'^{-1}(j)}\Big) \mathtt{des_X}(\sigma^{-1})+\sum_{\{\sigma \in \mathcal{S}_n\,|\,\tau' \sigma \in \mathcal{S}_n^{(j+1),j} \}} \Big(\mathtt{des_X}(\tau'\sigma) - X_{\sigma^{-1}\tau'^{-1}(j+1)}\Big) \mathtt{des_X}(\sigma^{-1})= \end{displaymath}

\begin{displaymath}\Pi_n(\tau')+ \sum_{\{\sigma \in \mathcal{S}_n\,|\,\tau' \sigma \in \mathcal{S}_n^{j,(j+1)} \}} X_{\sigma^{-1}\tau'^{-1}(j)}  \mathtt{des_X}(\sigma^{-1})- \sum_{\{\sigma \in \mathcal{S}_n\,|\,\tau' \sigma \in \mathcal{S}_n^{(j+1),j} \}} X_{\sigma^{-1}\tau'^{-1}(j+1)} \mathtt{des_X}(\sigma^{-1})= \end{displaymath}

\begin{displaymath}\Pi_n(\tau')+ \sum_{\sigma \in \mathcal{S}_n^{\tau'^{-1}(j),\tau'^{-1}(j+1)}} X_{\sigma^{-1}\tau'^{-1}(j)}  \mathtt{des_X}(\sigma^{-1})- \sum_{\sigma \in \mathcal{S}_n^{\tau'^{-1}(j+1),\tau'^{-1}(j)}} X_{\sigma^{-1}\tau'^{-1}(j+1)} \mathtt{des_X}(\sigma^{-1}).\end{displaymath}

\begin{itemize}

\item If \begin{math}\tau' \in \mathcal{S}_n^{j\,-\,(j+1)}\end{math} i.e. \begin{math}\tau'^{-1}(j+1)-\tau'^{-1}(j)>1\end{math} and \begin{math}\mathtt{des_X}\Big((j\ j+1)\tau'\Big)= \mathtt{des_X}(\tau')\end{math}, then 
\begin{displaymath}\sum_{\sigma \in \mathcal{S}_n^{\tau'^{-1}(j),\tau'^{-1}(j+1)}} X_{\sigma^{-1}\tau'^{-1}(j)} \mathtt{des_X}(\sigma^{-1})= \sum_{\sigma \in \mathcal{S}_n^{\tau'^{-1}(j+1),\tau'^{-1}(j)}} X_{\sigma^{-1}\tau'^{-1}(j+1)} \mathtt{des_X}(\sigma^{-1})\end{displaymath} and
\begin{displaymath}\Pi_n\big((j\ j+1) \tau'\big)=\Pi_n(\tau') =  \Pi_n(\iota) -  \phi\Big((j\ j+1)\tau'\Big)(n-2)!\mathfrak{d_n}.\end{displaymath}

\item If \begin{math}\tau' \in \mathcal{S}_n^{(j+1)\,-\,j}\end{math} i.e. \begin{math}\tau'^{-1}(j)-\tau'^{-1}(j+1)>1\end{math} and \begin{math}\mathtt{des_X}\Big((j\ j+1)\tau'\Big)=\mathtt{des_X}(\tau')\end{math}, then 
\begin{displaymath}\sum_{\sigma \in \mathcal{S}_n^{\tau'^{-1}(j),\tau'^{-1}(j+1)}} X_{\sigma^{-1}\tau'^{-1}(j)} \mathtt{des_X}(\sigma^{-1})= \sum_{\sigma \in \mathcal{S}_n^{\tau'^{-1}(j+1),\tau'^{-1}(j)}} X_{\sigma^{-1}\tau'^{-1}(j+1)} \mathtt{des_X}(\sigma^{-1})\end{displaymath}
and, once again, we have \begin{displaymath}\Pi_n\big((j\ j+1) \tau'\big)=\Pi_n(\tau') =  \Pi_n(\iota) -  \phi\Big((j\ j+1)\tau'\Big)(n-2)!\mathfrak{d_n}.\end{displaymath}

\item If \begin{math}\tau' \in \mathcal{S}_n^{j,(j+1)}\end{math} i.e. \begin{math}\tau'^{-1}(j+1)-\tau'^{-1}(j)=1\end{math} and \begin{math}\mathtt{des_X}\Big((j\ j+1)\tau'\Big)= \mathtt{des_X}(\tau')+X_{\tau'^{-1}(j)}\end{math}, then 
\begin{displaymath}\sum_{\sigma \in \mathcal{S}_n^{\tau'^{-1}(j),\tau'^{-1}(j+1)}} X_{\sigma^{-1}\tau'^{-1}(j)} \mathtt{des_X}(\sigma^{-1})= \sum_{\sigma \in \mathcal{S}_n^{\tau'^{-1}(j+1),\tau'^{-1}(j)}} X_{\sigma^{-1}\tau'^{-1}(j+1)} \mathtt{des_X}(\sigma^{-1})\ -X_{\tau'^{-1}(j)}(n-2)!\mathfrak{d_n}\end{displaymath} and
\begin{displaymath}\Pi_n\big( (j\ j+1) \tau' \big)=\Pi_n\big(\tau' \big)  -X_{\tau'^{-1}(j)}(n-2)!\mathfrak{d_n}= \Pi_n(\iota)- \phi\Big((j\ j+1)\tau'\Big)(n-2)!\mathfrak{d_n}. \end{displaymath}

\item If \begin{math}\tau' \in \mathcal{S}_n^{(j+1),j}\end{math} i.e. \begin{math}\tau'^{-1}(j)-\tau'^{-1}(j+1)=1\end{math} and \begin{math}\mathtt{des_X}\Big((j\ j+1)\tau'\Big)= \mathtt{des_X}(\tau')- X_{\tau'^{-1}(j+1)}\end{math}, then 
\begin{displaymath}\sum_{\sigma \in \mathcal{S}_n^{\tau'^{-1}(j),\tau'^{-1}(j+1)}}X_{\sigma^{-1}\tau'^{-1}(j)} \mathtt{des_X}(\sigma^{-1})=\sum_{\sigma \in \mathcal{S}_n^{\tau'^{-1}(j+1),\tau'^{-1}(j)}} X_{\sigma^{-1}\tau'^{-1}(j+1)} \mathtt{des_X}(\sigma^{-1})\ + X_{\tau'^{-1}(j+1)}(n-2)!\mathfrak{d_n}\end{displaymath} and
\begin{displaymath}\Pi_n\big( (j\ j+1) \tau' \big)=\Pi_n\big(\tau' \big)  +X_{\tau'^{-1}(j)}(n-2)!\mathfrak{d_n}= \Pi_n(\iota)- \phi\Big((j\ j+1)\tau'\Big)(n-2)!\mathfrak{d_n}. \end{displaymath}
\end{itemize}

\noindent Finally, we have
\begin{equation*}\left. \begin{array}{ccc}
\Pi_n(\tau) & = & \Pi_n(\iota) - \phi(\tau)(n-2)!\mathfrak{d_n}\\
& = & \Pi_n(\iota) - \big(\Big(\mathtt{des_X}(\tau)- \mathtt{des_X}(\tau_{k-1} \dots \tau_2 \tau_1 \iota)\Big)+ \phi\Big(\tau_{k-1} \dots \tau_2 \tau_1 \iota \Big)\big)(n-2)!\mathfrak{d_n}\\
& = & \Pi_n(\iota) - \Big(\mathtt{des_X}(\tau)- \mathtt{des_X}(\iota)\Big)(n-2)!\mathfrak{d_n}\\
& = & \Pi_n(\iota) - \mathtt{des_X}(\tau)(n-2)!\mathfrak{d_n}.
\end{array} \right. \end{equation*}
\end{proof}

\noindent Now we are in position to determine the minimal polynomial of \begin{math}\mathfrak{D_n}\end{math}.\\
We write \begin{math}\mathrm{I}_{n!}\end{math} for the matrix identity of dimension \begin{math}n!\end{math}.

\newtheorem{XMGrKoro}[XMNFijI]{Proposition}
\begin{XMGrKoro}
Let \begin{math}n \geq 3\end{math}. The minimal polynomial of \begin{math}\mathfrak{D_n}\end{math} is:  \begin{displaymath}X \big(X- \frac{n!}{2}\mathfrak{d_n}\big) \big(X+ (n-2)!\mathfrak{d_n}\big).\end{displaymath}
\end{XMGrKoro}

\begin{proof} Using Lemma 2.3 and Lemma 3.2, we get:
\begin{displaymath}\Big(\mathfrak{D_n} -  \frac{n!}{2}\mathfrak{d_n}\mathrm{I}_{n!} \Big)\mathfrak{D_n} \Big(\mathfrak{D_n}+ (n-2)!\mathfrak{d_n}\mathrm{I}_{n!}\Big)=\Big(\mathfrak{D_n} -  \frac{n!}{2}\mathfrak{d_n}\mathrm{I}_{n!} \Big) \Big(\Pi_n(\iota) \Big)_{\pi, \tau \in \mathcal{S}_n} = \Big(0\Big)_{\pi, \tau \in \mathcal{S}_n}.\end{displaymath}
Hence the minimal polynomial of \begin{math}\mathfrak{D_n}\end{math} divides \begin{math}X \big(X- \frac{n!}{2}\mathfrak{d_n}\big) \big(X+ (n-2)!\mathfrak{d_n}\big)\end{math}.\\
Since \begin{math}\Big((\mathfrak{D_n}- \frac{n!}{2}\mathfrak{d_n} \mathrm{I}_{n!})\mathfrak{D_n}\Big)_{\pi, \pi}=\Pi_n\end{math} resp. \begin{math}\Big(\mathfrak{D_n} \big(\mathfrak{D_n}+ (n-2)!\mathfrak{d_n}\mathrm{I}_{n!}\big)\Big)_{\pi, \pi}=\Pi_n\end{math}, then \begin{math}(X- \frac{n!}{2}\mathfrak{d_n})X\end{math} resp. \begin{math}X \big(X+ (n-2)!\mathfrak{d_n} \big)\end{math} is not the minimal polynomial of \begin{math}\mathfrak{D_n}\end{math}.\\
 We have  \begin{displaymath}\Big(\big(\mathfrak{D_n}- \frac{n!}{2}\mathfrak{d_n} \mathrm{I}_{n!}\big) \big(\mathfrak{D_n}+ (n-2)!\mathfrak{d_n}\mathrm{I}_{n!}\big)\Big)_{\pi, \iota}= \Pi_n - \mathtt{des_X}(\pi)\frac{n!}{2}\mathfrak{d_n}.\end{displaymath}  Since there are \begin{math}\pi, \tau \in \mathcal{S}_n\end{math} such that \begin{math}\mathtt{des_X}(\pi) \neq \mathtt{des_X}(\tau)\end{math}, then \begin{displaymath}\Big(\big(\mathfrak{D_n}- \frac{n!}{2}\mathfrak{d_n} \mathrm{I}_{n!}\big) \big(\mathfrak{D_n}+ (n-2)!\mathfrak{d_n}\mathrm{I}_{n!}\big)\Big)_{\pi, \tau \in \mathcal{S}_n} \neq (0)_{\pi, \tau \in \mathcal{S}_n},\end{displaymath} and \begin{math}\big(X- \frac{n!}{2}\mathfrak{d_n}\big) \big(X+ (n-2)!\mathfrak{d_n}\big)\end{math} is not the minimal polynomial of \begin{math}\mathfrak{D_n}\end{math}.\\
We conclude that the minimal polynomial of \begin{math}\mathfrak{D_n}\end{math} is \begin{math}X \big(X- \frac{n!}{2}\mathfrak{d_n}\big) \big(X+ (n-2)!\mathfrak{d_n}\big)\end{math}.  
\end{proof}

\section{Minimal Polynomial of the Multinomial Inversion Statistic}
\label{sec:inversion}

\noindent To determine the minimal polynomial of \begin{math}\mathfrak{I_n}\end{math} we need to prove several multinomial formulas:

\noindent Let \begin{math}n \geq 4\end{math} and \begin{math}i_1,j_1,i_2,j_2,a,b \in [n]\end{math}. We write:
\begin{displaymath}\mathcal{S}_n^{i_1\,j_1}:= \{ \sigma \in \mathcal{S}_n\ |\ \sigma(i_1) > \sigma(j_1)\},\end{displaymath}
\begin{displaymath}\mathcal{S}_n^{i_1\,j_1}{}_{|\bar{i_2}=a,\bar{j_2}=b}:= \{ \sigma \in \mathcal{S}_n^{i_1\,j_1}\ |\ \sigma^{-1}(i_2)=a,\ \sigma^{-1}(j_2)=b \}.\end{displaymath}

\newtheorem{drei}[XMNFijI]{Lemma}
\begin{drei}
Let \begin{math}n \geq 4\end{math} and \begin{math}i_1, j_1, i_2, j_2, a, b \in [n]\end{math}, with \begin{math}i_1 < j_1\end{math}, \begin{math}i_2 < j_2\end{math}. We have:
\begin{itemize}
\item if \begin{math}\{a, b\} = \{i_1, j_1\}\end{math}:
\begin{displaymath}\#\mathcal{S}_n^{i_1\,j_1}{}_{|\bar{i_2}=j_1,\bar{j_2}=i_1}  =  \#\mathcal{S}_n^{j_1\,i_1}{}_{|\bar{i_2}=i_1,\bar{j_2}=j_1}= (n-2)!.\end{displaymath}
\item if \begin{math}a,b \notin \{i_1, j_1\}\end{math}: 
\begin{displaymath}\#\mathcal{S}_n^{i_1\,j_1}{}_{|\bar{i_2}=a,\bar{j_2}=b}  = \#\mathcal{S}_n^{j_1\,i_1}{}_{|\bar{i_2}=a,\bar{j_2}=b} = \frac{(n-2)!}{2}.\end{displaymath}
\item if \begin{math}a=i_1\end{math} and \begin{math}b \notin \{i_1, j_1\}\end{math}:
\begin{equation*}\left. \begin{array}{ccc}
\#\mathcal{S}_n^{i_1\,j_1}{}_{|\bar{i_2}=i_1,\bar{j_2}=b} & = & (i_2-1)(n-3)!,\\ &&\\
\#\mathcal{S}_n^{j_1\,i_1}{}_{|\bar{i_2}=i_1,\bar{j_2}=b} & = & (n-i_2-1)(n-3)!.
\end{array} \right. \end{equation*}
\item if \begin{math}a \notin \{i_1, j_1\}\end{math} and \begin{math}b=i_1\end{math}:
\begin{equation*}\left. \begin{array}{ccc}
\#\mathcal{S}_n^{i_1\,j_1}{}_{|\bar{i_2}=a,\bar{j_2}=i_1} & = & (j_2-2)(n-3)!,\\ &&\\
\#\mathcal{S}_n^{j_1\,i_1}{}_{|\bar{i_2}=a,\bar{j_2}=i_1} & = & (n-j_2)(n-3)!.
\end{array} \right. \end{equation*}
\item if \begin{math}a=j_1\end{math} and \begin{math}b \notin \{i_1, j_1\}\end{math}:
\begin{equation*}\left. \begin{array}{ccc}
\#\mathcal{S}_n^{i_1\,j_1}{}_{|\bar{i_2}=j_1,\bar{j_2}=b} & = & (n-i_2-1)(n-3)!,\\ &&\\
\#\mathcal{S}_n^{j_1\,i_1}{}_{|\bar{i_2}=j_1,\bar{j_2}=b} & = &  (i_2-1)(n-3)!.
\end{array} \right. \end{equation*}
\item if \begin{math}a \notin \{i_1, j_1\}\end{math} and \begin{math}b=j_1\end{math}:
\begin{equation*}\left. \begin{array}{ccc}
\#\mathcal{S}_n^{i_1\,j_1}{}_{|\bar{i_2}=a,\bar{j_2}=j_1} & = & (n-j_2)(n-3)!,\\ &&\\
\#\mathcal{S}_n^{j_1\,i_1}{}_{|\bar{i_2}=a,\bar{j_2}=j_1} & = & (j_2-2)(n-3)!.
\end{array} \right. \end{equation*}
\end{itemize}
\end{drei}

\begin{proof} Cardinality calculating. \end{proof}

\noindent Let \begin{math}n \geq 4\end{math} and \begin{math}i,j \in [n]\end{math}. We write
\begin{displaymath}{}^{i\,j}\mathcal{S}_n:= \{ \sigma \in \mathcal{S}_n\ |\ \sigma^{-1}(i) > \sigma^{-1}(j)\},\end{displaymath}
\begin{displaymath}\chi_{i,j}=j-i-1.\end{displaymath}

\newtheorem{prel}[XMNFijI]{Lemma}
\begin{prel}
Let \begin{math}n \geq 4\end{math} and \begin{math}i_1, j_1, i_2, j_2 \in [n]\end{math} with $i_1 < j_1$ and \begin{math}i_2 < j_2\end{math}. Then: 
\begin{itemize}
\item[(a)]\begin{displaymath}\#{}^{i_1\,j_1}\mathcal{S}_n \cap \mathcal{S}_n^{i_2\,j_2} = (n-3)!\chi_{i_1,j_1} \chi_{i_2,j_2} + (n-2)!\frac{n^2-n+2}{4},\end{displaymath}
\item[(b)]\begin{displaymath}\#{}^{i_i\,j_1}\mathcal{S}_n \cap \mathcal{S}_n^{j_2\,i_2} =  (n-2)!\frac{n^2-n-2}{4} - (n-3)!\chi_{i_1,j_1} \chi_{i_2,j_2}.\end{displaymath} 
\end{itemize}
\end{prel}

\begin{proof} \begin{itemize}
\item[(a)] \begin{equation*}\left. \begin{array}{ccc} 
\#{}^{i_1\,j_1}\mathcal{S}_n \cap \mathcal{S}_n^{i_2\,j_2} & = & \#\{\sigma \in \mathcal{S}_n\ |\ \sigma(i_2)=j_1, \sigma(j_2)=i_1\} \\
& +  &  \#\{\sigma \in \mathcal{S}_n\ |\ \sigma(i_2)=j_1, \sigma(j_2) \neq i_1, \sigma^{-1}(j_1) < \sigma^{-1}(i_1),\sigma(i_2) > \sigma(j_2)\}  \\
& +  &   \#\{\sigma \in \mathcal{S}_n\ |\ \sigma(i_2) \neq j_1, \sigma(j_2)=i_1, \sigma^{-1}(j_1) < \sigma^{-1}(i_1), \sigma(i_2) > \sigma(j_2)\}  \\
& +  &   \#\{\sigma \in \mathcal{S}_n\ |\ \sigma(j_2)=j_1, \sigma^{-1}(j_1) < \sigma^{-1}(i_1), \sigma(i_2) > \sigma(j_2)\} \\
& + &   \#\{\sigma \in \mathcal{S}_n\ |\ \sigma(i_2)=i_1, \sigma^{-1}(j_1) < \sigma^{-1}(i_1), \sigma(i_2) > \sigma(j_2)\}  \\
& +  &  \#\{\sigma \in \mathcal{S}_n\ |\ \sigma(i_2) \notin \{i_1,j_1\}, \sigma(j_2) \notin \{i_1,j_1\}, \sigma^{-1}(j_1) < \sigma^{-1}(i_1), \sigma(i_2) > \sigma(j_2)\}  \\
& = &  (n-2)!  \\
& + &   (j_1-2)(n-i_2-1)(n-3)! \\
& + &   (n-i_1-1)(j_2-2)(n-3)! \\
& + &   (n-j_1)(n-j_2)(n-3)! \\
& + & (i_1-1)(i_2-1)(n-3)!  \\ 
& + & \binom{n-2}{2}\binom{n-2}{2}(n-4)!  \\
& = &  (n-3)!\Big((i_1-j_1 +1)(i_2-j_2 +1) + \frac{n^3-3n^2+4n-4}{4}\Big).   
\end{array} \right. \end{equation*}

\item[(b)]  \begin{equation*}\left. \begin{array}{ccc}
\#{}^{i_i\,j_1}\mathcal{S}_n \cap \mathcal{S}_n^{j_2\,i_2} & = &  \#\{\sigma \in \mathcal{S}_n\ |\ \sigma(i_2)=i_1, \sigma^{-1}(j_1) < \sigma^{-1}(i_1),\sigma(i_2) < \sigma(j_2)\}  \\
& + &  \#\{\sigma \in \mathcal{S}_n\ |\ \sigma(j_2)=j_1, \sigma^{-1}(j_1) < \sigma^{-1}(i_1), \sigma(i_2) < \sigma(j_2)\} \\
& + &  \#\{\sigma \in \mathcal{S}_n\ |\ \sigma(i_2) \neq j_1, \sigma(j_2)=i_1, \sigma^{-1}(j_1) < \sigma^{-1}(i_1), \sigma(i_2) < \sigma(j_2)\}  \\
& + &  \#\{\sigma \in \mathcal{S}_n\ |\ \sigma(i_2)=j_1, \sigma(j_2) \neq i_1, \sigma^{-1}(j_1) < \sigma^{-1}(i_1), \sigma(i_2) < \sigma(j_2)\}  \\
& + &  \#\{\sigma \in \mathcal{S}_n\ |\ \sigma(i_2) \notin \{i_1,j_1\}, \sigma(j_2) \notin \{i_1,j_1\}, \sigma^{-1}(j_1) < \sigma^{-1}(i_1), \sigma(i_2) < \sigma(j_2)\} \\
& = &  (n-i_1-1)(i_2-1)(n-3)! \\
& + & (j_1-2)(n-j_2)(n-3)!  \\
& + &  (i_1 -1)(j_2 -2)(n-3)!  \\
& + &  (n- j_1)(n-i_2-1)(n-3)! \\
& + & \binom{n-2}{2}\binom{n-2}{2}(n-4)! \\
& = &  (n-3)!\Big((i_1-j_1 +1)(-i_2+j_2 -1) + \frac{(n+1)(n-2)^2}{4}\Big). 
\end{array} \right. \end{equation*}
\end{itemize}
\end{proof}

\noindent Let \begin{math}n \geq 4\end{math}. We define: 
\begin{displaymath}\left.\begin{array}{cccc}
\mathtt{f_2}: & \mathcal{S}_n & \rightarrow & \mathbb{R}[X_{1,2}, \dots, X_{n-1,n}]\\
   & \pi & \mapsto & \sum_{\sigma \in \mathcal{S}_n} \mathtt{inv_X}(\sigma^{-1}) \mathtt{inv_X}(\sigma \pi) 
\end{array}\right.\end{displaymath}

\newtheorem{invinv}[XMNFijI]{Lemma}
\begin{invinv}
Let \begin{math}n \geq 4\end{math}. Then: 
\begin{itemize}
\item[(a)] For \begin{math}(i_1,j_1)=(i_2,j_2)\end{math}:
\begin{itemize}
\item[\begin{math}\triangleright\end{math}] If \begin{math} \pi(i_1)<\pi(j_1) \end{math}:
\begin{displaymath}[X_{i_1,j_1}^2]\mathtt{f_2}(\pi)=  (n-3)!\chi_{i_1,j_1} \chi_{\pi(i_1),\pi(j_1)} + (n-2)!\frac{n^2-n+2}{4}.  \end{displaymath}
\item[\begin{math}\triangleright\end{math}] If \begin{math} \pi(i_1)>\pi(j_1) \end{math}:
\begin{displaymath}[X_{i_1,j_1}^2]\mathtt{f_2}(\pi) =  (n-2)!\frac{n^2-n-2}{4} - (n-3)!\chi_{i_1,j_1} \chi_{\pi(j_1),\pi(i_1)}. \end{displaymath}
\end{itemize}

\item[(b)] For \begin{math}(i_1,j_1) \neq (i_2,j_2)\end{math}:
\begin{itemize}
\item[\begin{math}\triangleright\end{math}] If \begin{math}\pi(i_1) < \pi(j_1)\end{math} and \begin{math}\pi(i_2) < \pi(j_2)\end{math}:
\begin{displaymath}[X_{i_1,j_1}X_{i_2,j_2}]\mathtt{f_2}(\pi) = (n-3)!\big(\chi_{i_1,j_1} \chi_{\pi(i_2),\pi(j_2)} + \chi_{\pi(i_1),\pi(j_1)} \chi_{i_2,j_2} \big)+ (n-2)!\frac{n^2-n+2}{2}.\end{displaymath}
\item[\begin{math}\triangleright\end{math}] If \begin{math}\pi(i_1) < \pi(j_1)\end{math} and \begin{math}\pi(i_2) > \pi(j_2)\end{math}:
\begin{displaymath}[X_{i_1,j_1}X_{i_2,j_2}]\mathtt{f_2}(\pi) = (n-3)!\big(\chi_{\pi(i_1),\pi(j_1)} \chi_{i_2,j_2} - \chi_{i_1,j_1} \chi_{\pi(j_2),\pi(i_2)} \big)+ \frac{n!}{2}.\end{displaymath}
\item[\begin{math}\triangleright\end{math}] If \begin{math}\pi(i_1) > \pi(j_1)\end{math} and \begin{math}\pi(i_2) > \pi(j_2)\end{math}:
\begin{displaymath}[X_{i_1,j_1}X_{i_2,j_2}]\mathtt{f_2}(\pi) = (n-2)!\frac{n^2-n-2}{2} - (n-3)!\big(\chi_{i_1,j_1} \chi_{\pi(j_2),\pi(i_2)} + \chi_{\pi(j_1),\pi(i_1)} \chi_{i_2,j_2} \big).\end{displaymath}
\end{itemize}
\end{itemize}
\end{invinv}

\begin{proof}
\begin{itemize}
\item[(a)] For \begin{math}(i_1,j_1)=(i_2,j_2)\end{math}: \begin{displaymath}[X_{i_1,j_1}^2]\mathtt{f_2}(\pi)= \#{}^{i_1\,j_1}\mathcal{S}_n \cap \mathcal{S}_n^{\pi(i_1)\,\pi(j_1)}.\end{displaymath}
\item[(b)] For \begin{math}(i_1,j_1) \neq (i_2,j_2)\end{math}: \begin{displaymath}[X_{i_1,j_1}X_{i_2,j_2}]\mathtt{f_2}(\pi)= \#{}^{i_1\,j_1}\mathcal{S}_n \cap \mathcal{S}_n^{\pi(i_2)\,\pi(j_2)} + \#{}^{i_2\,j_2}\mathcal{S}_n \cap \mathcal{S}_n^{\pi(i_1)\,\pi(j_1)}.\end{displaymath}
\end{itemize}
\end{proof}

\newtheorem{prel3}[XMNFijI]{Lemma}
\begin{prel3}
Let \begin{math}n \geq 4\end{math} and \begin{math}i_1, j_1, i_2, j_2, i_3, j_3 \in [n]\end{math} with \begin{math}i_1 < j_1\end{math}, \begin{math}i_2 < j_2\end{math} and \begin{math}i_3 < j_3\end{math}. Then:
\begin{itemize}
\item[(a)] \begin{displaymath}\sum_{a,b \in [n]} \# {}^{i_1\,j_1}\mathcal{S}_n \cap \mathcal{S}_n^{a\,b} \times \mathcal{S}_n^{i_3\,j_3}{}_{|\bar{i_2}=a,\bar{j_2}=b} = \frac{n!^2}{8} - \frac{(n-2)!^2}{2} - (n-4)(n-3)!^2\chi_{i_1,j_1} \chi_{i_2,j_2} \chi_{i_3,j_3}\end{displaymath}
\begin{displaymath}- (n-3)!(n-2)! \Big(\chi_{i_1,j_1} \chi_{i_2,j_2} + \chi_{i_1,j_1} \chi_{i_3,j_3} + \chi_{i_2,j_2} \chi_{i_3,j_3} \Big).\end{displaymath}
\item[(b)] \begin{displaymath}\sum_{a,b \in [n]} \# {}^{i_1\,j_1}\mathcal{S}_n \cap \mathcal{S}_n^{a\,b} \times \mathcal{S}_n^{j_3\,i_3}{}_{|\bar{i_2}=a,\bar{j_2}=b} = \frac{n!^2}{8} + \frac{(n-2)!^2}{2} + (n-4)(n-3)!^2\chi_{i_1,j_1} \chi_{i_2,j_2} \chi_{i_3,j_3}\end{displaymath}
\begin{displaymath}+ (n-3)!(n-2)! \Big(\chi_{i_1,j_1} \chi_{i_2,j_2} + \chi_{i_1,j_1} \chi_{i_3,j_3} + \chi_{i_2,j_2} \chi_{i_3,j_3} \Big).\end{displaymath}
\end{itemize} 
\end{prel3}

\begin{proof} Using Lemma 4.1 and Lemma 4.2, we get:
\begin{itemize}
\item[(a)]  \begin{equation*}\left. \begin{array}{ccc}
\sum_{a,b \in [n]} {}^{i_1\,j_1}\mathcal{S}_n \cap \mathcal{S}_n^{a\,b} \times \mathcal{S}_n^{i_3\,j_3}{}_{|\bar{i_2}=a,\bar{j_2}=b} & = & \# {}^{i_1\,j_1}\mathcal{S}_n \cap \mathcal{S}_n^{j_3\,i_3} \times \mathcal{S}_n^{i_3\,j_3}{}_{|\bar{i_2}=j_3,\bar{j_2}=i_3}\\
& + & \# \bigcup_{a,b \notin \{i_3,j_3\}} {}^{i_1\,j_1}\mathcal{S}_n \cap \mathcal{S}_n^{a\,b} \times \mathcal{S}_n^{i_3\,j_3}{}_{|\bar{i_2}=a,\bar{j_2}=b} \\
& + & \# \bigcup_{b < i_3} {}^{i_1\,j_1}\mathcal{S}_n \cap \mathcal{S}_n^{i_3\,b} \times \mathcal{S}_n^{i_3\,j_3}{}_{|\bar{i_2}=i_3,\bar{j_2}=b} \\
& + & \# \bigcup_{\begin{subarray}{l} b > i_3 \\ b \neq j_3 \end{subarray}} {}^{i_1\,j_1}\mathcal{S}_n \cap \mathcal{S}_n^{i_3\,b} \times \mathcal{S}_n^{i_3\,j_3}{}_{|\bar{i_2}=i_3,\bar{j_2}=b}\\
& + & \# \bigcup_{\begin{subarray}{l} b < j_3 \\ b \neq i_3 \end{subarray}} {}^{i_1\,j_1}\mathcal{S}_n \cap \mathcal{S}_n^{i_3\,b} \times \mathcal{S}_n^{i_3\,j_3}{}_{|\bar{i_2}=i_3,\bar{j_2}=b}\\
& + & \# \bigcup_{b > j_3} {}^{i_1\,j_1}\mathcal{S}_n \cap \mathcal{S}_n^{j_3\,b} \times \mathcal{S}_n^{i_3\,j_3}{}_{|\bar{i_2}=j_3,\bar{j_2}=b}\\
& + & \# \bigcup_{a < i_3} {}^{i_1\,j_1}\mathcal{S}_n \cap \mathcal{S}_n^{a\,i_3} \times \mathcal{S}_n^{i_3\,j_3}{}_{|\bar{i_2}=a,\bar{j_2}=i_3} \\
& + & \#\bigcup_{\begin{subarray}{l} a > i_3 \\ a \neq j_3 \end{subarray}} {}^{i_1\,j_1}\mathcal{S}_n \cap \mathcal{S}_n^{a\,i_3} \times \mathcal{S}_n^{i_3\,j_3}{}_{|\bar{i_2}=a,\bar{j_2}=i_3}\\
& + & \# \bigcup_{\begin{subarray}{l} a < j_3 \\ a \neq i_3 \end{subarray}} {}^{i_1\,j_1}\mathcal{S}_n \cap \mathcal{S}_n^{a\,j_3} \times \mathcal{S}_n^{i_3\,j_3}{}_{|\bar{i_2}=a,\bar{j_2}=j_3}\\
& + & \# \bigcup_{a > j_3} {}^{i_1\,j_1}\mathcal{S}_n \cap \mathcal{S}_n^{a\,j_3} \times \mathcal{S}_n^{i_3\,j_3}{}_{|\bar{i_2}=a,\bar{j_2}=j_3}
\end{array} \right. \end{equation*}
\begin{equation*}\left. \begin{array}{cc}
= &  (n-2)!^2\frac{n^2-n-2}{4} - (n-3)!(n-2)!\chi_{i_1,j_1} \chi_{i_3,j_3}\\
+ &  \binom{n-2}{2}\frac{(n-2)!n!}{4} \\
+ &  (i_2-1)(i_3-1)(n-3)!(n-2)!\frac{n^2-n-2}{4} - (i_2-1)\binom{i_3-1}{2}(n-3)!^2\chi_{i_1,j_1}\\
+ &  (i_2-1)\binom{n-i_3}{2}(n-3)!^2\chi_{i_1,j_1} - (i_2-1)(n-3)!^2\chi_{i_1,j_1}\chi_{i_3,j_3}\\
& + (i_2-1)(n-i_3-1)(n-3)!(n-2)!\frac{n^2-n+2}{4}\\ 
+ &  (n-i_2-1)(j_3-2)(n-3)!(n-2)!\frac{n^2-n-2}{4} - (n-i_2-1)\binom{j_3-1}{2}(n-3)!^2\chi_{i_1,j_1}\\
& + (n-i_2-1)(n-3)!^2\chi_{i_1,j_1}\chi_{i_3,j_3}\\
+ &  (n-i_2-1)\binom{n-j_3}{2}(n-3)!^2\chi_{i_1,j_1} + (n-i_2-1)(n-j_3)(n-3)!(n-2)!\frac{n^2-n+2}{4}\\
+ &  (j_2-2)\binom{i_3-1}{2}(n-3)!^2\chi_{i_1,j_1} + (j_2-2)(i_3-1)(n-3)!(n-2)!\frac{n^2-n+2}{4}\\
+ &  (j_2-2)(n-i_3-1)(n-3)!(n-2)!\frac{n^2-n-2}{4} - (j_2-2)\binom{n-i_3}{2}(n-3)!^2\chi_{i_1,j_1}\\
& + (j_2-2)(n-3)!^2\chi_{i_1,j_1}\chi_{i_3,j_3}\\
+ &  (n-j_2)\binom{j_3-1}{2}(n-3)!^2\chi_{i_1,j_1} - (n-j_2)(n-3)!^2\chi_{i_1,j_1}\chi_{i_3,j_3}\\
& + (n-j_2)(j_3-2)(n-3)!(n-2)!\frac{n^2-n+2}{4}\\
+ &  (n-j_2)(n-j_3)(n-3)!(n-2)!\frac{n^2-n-2}{4} - (n-j_2)\binom{n-j_3}{2}(n-3)!^2\chi_{i_1,j_1}\\
= & \frac{n!^2}{8} - \frac{(n-2)!^2}{2} - (n-4)(n-3)!^2\chi_{i_1,j_1} \chi_{i_2,j_2} \chi_{i_3,j_3}\\
& - (n-3)!(n-2)! \Big(\chi_{i_1,j_1} \chi_{i_2,j_2} + \chi_{i_1,j_1} \chi_{i_3,j_3} + \chi_{i_2,j_2} \chi_{i_3,j_3} \Big).
\end{array} \right. \end{equation*}

\item[(b)]  \begin{equation*}\left. \begin{array}{ccc}
\sum_{a,b \in [n]} {}^{i_1\,j_1}\mathcal{S}_n \cap \mathcal{S}_n^{a\,b} \times \mathcal{S}_n^{j_3\,i_3}{}_{|\bar{i_2}=a,\bar{j_2}=b} & = &  \# {}^{i_1\,j_1}\mathcal{S}_n \cap \mathcal{S}_n^{i_3\,j_3} \times \mathcal{S}_n^{j_3\,i_3}{}_{|\bar{i_2}=i_3,\bar{j_2}=j_3} \\
& + &  \# \bigcup_{a,b \notin \{i_3,j_3\}} {}^{i_1\,j_1}\mathcal{S}_n \cap \mathcal{S}_n^{a\,b} \times \mathcal{S}_n^{j_3\,i_3}{}_{|\bar{i_2}=a,\bar{j_2}=b}   \\
& + &  \# \bigcup_{b < i_3} {}^{i_1\,j_1}\mathcal{S}_n \cap \mathcal{S}_n^{i_3\,b} \times \mathcal{S}_n^{j_3\,i_3}{}_{|\bar{i_2}=i_3,\bar{j_2}=b}   \\
& + & \# \bigcup_{\begin{subarray}{l} b > i_3 \\ b \neq j_3 \end{subarray}} {}^{i_1\,j_1}\mathcal{S}_n \cap \mathcal{S}_n^{i_3\,b} \times \mathcal{S}_n^{j_3\,i_3}{}_{|\bar{i_2}=i_3,\bar{j_2}=b}    \\
& + &  \# \bigcup_{\begin{subarray}{l} b < j_3 \\ b \neq i_3 \end{subarray}} {}^{i_1\,j_1}\mathcal{S}_n \cap \mathcal{S}_n^{i_3\,b} \times \mathcal{S}_n^{j_3\,i_3}{}_{|\bar{i_2}=i_3,\bar{j_2}=b}   \\
& + &  \# \bigcup_{b > j_3} {}^{i_1\,j_1}\mathcal{S}_n \cap \mathcal{S}_n^{j_3\,b} \times \mathcal{S}_n^{j_3\,i_3}{}_{|\bar{i_2}=j_3,\bar{j_2}=b}    \\
& + & \# \bigcup_{a < i_3} {}^{i_1\,j_1}\mathcal{S}_n \cap \mathcal{S}_n^{a\,i_3} \times \mathcal{S}_n^{j_3\,i_3}{}_{|\bar{i_2}=a,\bar{j_2}=i_3}    \\
& + &  \# \bigcup_{\begin{subarray}{l} a > i_3 \\ a \neq j_3 \end{subarray}} {}^{i_1\,j_1}\mathcal{S}_n \cap \mathcal{S}_n^{a\,i_3} \times \mathcal{S}_n^{j_3\,i_3}{}_{|\bar{i_2}=a,\bar{j_2}=i_3}   \\
& + &  \# \bigcup_{\begin{subarray}{l} a < j_3 \\ a \neq i_3 \end{subarray}} {}^{i_1\,j_1}\mathcal{S}_n \cap \mathcal{S}_n^{a\,j_3} \times \mathcal{S}_n^{j_3\,i_3}{}_{|\bar{i_2}=a,\bar{j_2}=j_3}   \\
& + &  \# \bigcup_{a > j_3} {}^{i_1\,j_1}\mathcal{S}_n \cap \mathcal{S}_n^{a\,j_3} \times \mathcal{S}_n^{j_3\,i_3}{}_{|\bar{i_2}=a,\bar{j_2}=j_3}   \\
\end{array} \right. \end{equation*}
\begin{equation*}\left. \begin{array}{cc}
= &  (n-3)!(n-2)!\chi_{i_1,j_1} \chi_{i_3,j_3} + (n-2)!^2\frac{n^2-n+2}{4}\\
+ & \binom{n-2}{2}\frac{(n-2)!n!}{4}  \\
+ &  (n-i_2-1)(i_3-1)(n-3)!(n-2)!\frac{n^2-n-2}{4} - (n-i_2-1)\binom{i_3-1}{2}(n-3)!^2\chi_{i_1,j_1}  \\
+ &  (n-i_2-1)\binom{n-i_3}{2}(n-3)!^2\chi_{i_1,j_1} - (n-i_2-1)(n-3)!^2\chi_{i_1,j_1}\chi_{i_3,j_3} \\
& + (n-i_2-1)(n-i_3-1)(n-3)!(n-2)!\frac{n^2-n+2}{4}  \\
+ &  (i_2-1)(j_3-2)(n-3)!(n-2)!\frac{n^2-n-2}{4} - (i_2-1)\binom{j_3-1}{2}(n-3)!^2\chi_{i_1,j_1}+ (i_2-1)(n-3)!^2\chi_{i_1,j_1}\chi_{i_3,j_3} \\
+ &  (i_2-1)\binom{n-j_3}{2}(n-3)!^2\chi_{i_1,j_1} + (i_2-1)(n-j_3)(n-3)!(n-2)!\frac{n^2-n+2}{4} \\
+ &   (n-j_2)\binom{i_3-1}{2}(n-3)!^2\chi_{i_1,j_1} + (n-j_2)(i_3-1)(n-3)!(n-2)!\frac{n^2-n+2}{4} \\
+ &  (n-j_2)(n-i_3-1)(n-3)!(n-2)!\frac{n^2-n-2}{4} - (n-j_2)\binom{n-i_3}{2}(n-3)!^2\chi_{i_1,j_1}\\
& + (n-j_2)(n-3)!^2\chi_{i_1,j_1}\chi_{i_3,j_3}  \\
+ &  (j_2-2)\binom{j_3-1}{2}(n-3)!^2\chi_{i_1,j_1} - (j_2-2)(n-3)!^2\chi_{i_1,j_1}\chi_{i_3,j_3}+ (j_2-2)(j_3-2)(n-3)!(n-2)!\frac{n^2-n+2}{4}  \\
+ &  (j_2-2)(n-j_3)(n-3)!(n-2)!\frac{n^2-n-2}{4} - (j_2-2)\binom{n-j_3}{2}(n-3)!^2\chi_{i_1,j_1}  \\
= &  \frac{n!^2}{8} + \frac{(n-2)!^2}{2} + (n-4)(n-3)!^2\chi_{i_1,j_1} \chi_{i_2,j_2} \chi_{i_3,j_3}  \\
 & + (n-3)!(n-2)! \Big(\chi_{i_1,j_1} \chi_{i_2,j_2} + \chi_{i_1,j_1} \chi_{i_3,j_3} + \chi_{i_2,j_2} \chi_{i_3,j_3} \Big)  
\end{array} \right. \end{equation*}
\end{itemize}
\end{proof}

\noindent Let \begin{math}n \geq 4\end{math}. We define: 
\begin{displaymath}\left.\begin{array}{cccc}
\mathtt{f_3}: & \mathcal{S}_n & \rightarrow & \mathbb{R}[X_{1,2}, \dots, X_{n-1,n}]\\
   & \pi & \mapsto & \sum_{\sigma, \tau \in \mathcal{S}_n} \mathtt{inv_X}(\sigma^{-1}) \mathtt{inv_X}(\sigma \tau^{-1}) \mathtt{inv_X}(\tau \pi).        
\end{array}\right.\end{displaymath}

\newtheorem{invinvinv}[XMNFijI]{Lemma}
\begin{invinvinv}
Let \begin{math}n \geq 4\end{math}. Then: 
\begin{itemize}
\item[(a)] For \begin{math}(i_1,j_1)=(i_2,j_2)=(i_3,j_3)\end{math}: 
\begin{itemize}
\item[\begin{math}\triangleright\end{math}] If \begin{math}\pi(i_1) < \pi(j_1)\end{math}: 
\begin{displaymath}[X_{i_1,j_1}^3]\mathtt{f_3}(\pi)= \frac{n!^2}{8} - \frac{(n-2)!^2}{2} - (n-4)(n-3)!^2\chi_{i_1,j_1}^2\chi_{\pi(i_1),\pi(j_1)}\end{displaymath} \begin{displaymath}-(n-3)!(n-2)!\chi_{i_1,j_1}^2 -2(n-3)!(n-2)!\chi_{i_1,j_1}\chi_{\pi(i_1),\pi(j_1)}.\end{displaymath}
\item[\begin{math}\triangleright\end{math}] If \begin{math}\pi(i_1) > \pi(j_1)\end{math}:
\begin{displaymath}[X_{i_1,j_1}^3]\mathtt{f_3}(\pi)= \frac{n!^2}{8} + \frac{(n-2)!^2}{2} + (n-4)(n-3)!^2\chi_{i_1,j_1}^2\chi_{\pi(j_1),\pi(i_1)}\end{displaymath} \begin{displaymath}+(n-3)!(n-2)!\chi_{i_1,j_1}^2 + 2(n-3)!(n-2)!\chi_{i_1,j_1}\chi_{\pi(j_1),\pi(i_1)}.\end{displaymath} 
\end{itemize}
\item[(b)] For \begin{math}(i_1,j_1)=(i_2,j_2) \neq (i_3,j_3)\end{math}:
\begin{itemize}
\item[\begin{math}\triangleright\end{math}] If \begin{math}\pi(i_1) < \pi(j_1)\end{math} and \begin{math}\pi(i_3) < \pi(j_3)\end{math}: 
\begin{displaymath}[X_{i_1,j_1}^2X_{i_3,j_3}]\mathtt{f_3}(\pi)= \frac{3n!^2}{8} - \frac{3(n-2)!^2}{2}- (n-4)(n-3)!^2\big(\chi_{i_1,j_1}^2\chi_{\pi(i_3),\pi(j_3)} + 2\chi_{i_1,j_1}\chi_{i_3,j_3}\chi_{\pi(i_1),\pi(j_1)} \big)\end{displaymath}  \begin{displaymath}- (n-3)!(n-2)!\big(\chi_{i_1,j_1}^2 + 2\chi_{i_1,j_1}\chi_{\pi(i_3),\pi(j_3)} + 2\chi_{i_1,j_1}\chi_{i_3,j_3} + 2\chi_{i_1,j_1}\chi_{\pi(i_1),\pi(j_1)} + 2\chi_{i_3,j_3}\chi_{\pi(i_1),\pi(j_1)}\big).\end{displaymath}
\item[\begin{math}\triangleright\end{math}] If \begin{math}\pi(i_1) < \pi(j_1)\end{math} and \begin{math}\pi(i_3) > \pi(j_3)\end{math}: 
\begin{displaymath}[X_{i_1,j_1}^2X_{i_3,j_3}]\mathtt{f_3}(\pi)= \frac{3n!^2}{8} - \frac{(n-2)!^2}{2}+ (n-4)(n-3)!^2\big( \chi_{i_1,j_1}^2\chi_{\pi(j_3),\pi(i_3)} - 2\chi_{i_1,j_1}\chi_{i_3,j_3}\chi_{\pi(i_1),\pi(j_1)} \big)\end{displaymath}
\begin{displaymath}+ (n-3)!(n-2)!\big(\chi_{i_1,j_1}^2 + 2\chi_{i_1,j_1}\chi_{\pi(j_3),\pi(i_3)} - 2\chi_{i_1,j_1}\chi_{i_3,j_3} - 2\chi_{i_1,j_1}\chi_{\pi(i_1),\pi(j_1)} - 2\chi_{i_3,j_3}\chi_{\pi(i_1),\pi(j_1)}\big).\end{displaymath}
\item[\begin{math}\triangleright\end{math}] If \begin{math}\pi(i_1) > \pi(j_1)\end{math} and \begin{math}\pi(i_3) < \pi(j_3)\end{math}:
\begin{displaymath}[X_{i_1,j_1}^2X_{i_3,j_3}]\mathtt{f_3}(\pi)= \frac{3n!^2}{8} + \frac{(n-2)!^2}{2}- (n-4)(n-3)!^2\big( \chi_{i_1,j_1}^2\chi_{\pi(i_3),\pi(j_3)} - 2\chi_{i_1,j_1}\chi_{i_3,j_3}\chi_{\pi(j_1),\pi(i_1)} \big)\end{displaymath}
\begin{displaymath}- (n-3)!(n-2)!\big(\chi_{i_1,j_1}^2 + 2\chi_{i_1,j_1}\chi_{\pi(i_3),\pi(j_3)} - 2\chi_{i_1,j_1}\chi_{i_3,j_3} - 2\chi_{i_1,j_1}\chi_{\pi(j_1),\pi(i_1)} - 2\chi_{i_3,j_3}\chi_{\pi(j_1),\pi(i_1)}\big).\end{displaymath}
\item[\begin{math}\triangleright\end{math}] If \begin{math}\pi(i_1) > \pi(j_1)\end{math} and \begin{math}\pi(i_3) > \pi(j_3)\end{math}:
\begin{displaymath}[X_{i_1,j_1}^2X_{i_3,j_3}\mathtt{f_3}(\pi)= \frac{3n!^2}{8} + \frac{3(n-2)!^2}{2}+ (n-4)(n-3)!^2\big(\chi_{i_1,j_1}^2\chi_{\pi(j_3),\pi(i_3)} + 2\chi_{i_1,j_1}\chi_{i_3,j_3}\chi_{\pi(j_1),\pi(i_1)} \big)\end{displaymath}  \begin{displaymath}+ (n-3)!(n-2)!\big(\chi_{i_1,j_1}^2 + 2\chi_{i_1,j_1}\chi_{\pi(j_3),\pi(i_3)} + 2\chi_{i_1,j_1}\chi_{i_3,j_3} + 2\chi_{i_1,j_1}\chi_{\pi(j_1),\pi(i_1)} + 2\chi_{i_3,j_3}\chi_{\pi(j_1),\pi(i_1)}\big).\end{displaymath}
\end{itemize}
\item[(c)] For \begin{math}(i_1,j_1) \neq (i_2,j_2) \neq (i_3,j_3)\end{math}:
\begin{itemize}
\item[\begin{math}\triangleright\end{math}] If \begin{math}\pi(i_1) < \pi(j_1)\end{math}, \begin{math}\pi(i_2) < \pi(j_2)\end{math} and \begin{math}\pi(i_3) < \pi(j_3)\end{math}:
\begin{displaymath}[X_{i_1,j_1}X_{i_2,j_2}X_{i_3,j_3}]\mathtt{f_3}(\pi)= \frac{3n!^2}{4} - 3(n-2)!^2\end{displaymath}
\begin{displaymath}- 2(n-4)(n-3)!^2\big(\chi_{i_1,j_1}\chi_{i_2,j_2}\chi_{\pi(i_3),\pi(j_3)} + \chi_{i_1,j_1}\chi_{\pi(i_2),\pi(j_2)}\chi_{i_3,j_3} + \chi_{\pi(i_1),\pi(j_1)}\chi_{i_2,j_2}\chi_{i_3,j_3}\big)\end{displaymath}
\begin{displaymath}- 2(n-3)!(n-2)!\big(\chi_{i_1,j_1}\chi_{i_2,j_2} + \chi_{i_1,j_1}\chi_{\pi(i_3),\pi(j_3)} + \chi_{i_2,j_2}\chi_{\pi(i_3),\pi(j_3)}\end{displaymath}
\begin{displaymath}+ \chi_{i_1,j_1}\chi_{i_3,j_3} + \chi_{i_1,j_1}\chi_{\pi(i_2),\pi(j_2)} + \chi_{i_3,j_3}\chi_{\pi(i_2),\pi(j_2)}+ \chi_{i_2,j_2}\chi_{i_3,j_3} + \chi_{i_2,j_2}\chi_{\pi(i_1),\pi(j_1)} + \chi_{i_3,j_3}\chi_{\pi(i_1),\pi(j_1)}  \big).\end{displaymath}
\item[\begin{math}\triangleright\end{math}] If \begin{math}\pi(i_1) < \pi(j_1)\end{math}, \begin{math}\pi(i_2) < \pi(j_2)\end{math} and \begin{math}\pi(i_3) > \pi(j_3)\end{math}: 
\begin{displaymath}[X_{i_1,j_1}X_{i_2,j_2}X_{i_3,j_3}]\mathtt{f_3}(\pi)= \frac{3n!^2}{4} - (n-2)!^2\end{displaymath}
\begin{displaymath}+ 2(n-4)(n-3)!^2\big(\chi_{i_1,j_1}\chi_{i_2,j_2}\chi_{\pi(j_3),\pi(i_3)} - \chi_{i_1,j_1}\chi_{\pi(i_2),\pi(j_2)}\chi_{i_3,j_3} - \chi_{\pi(i_1),\pi(j_1)}\chi_{i_2,j_2}\chi_{i_3,j_3}\big)\end{displaymath}
\begin{displaymath}+ 2(n-3)!(n-2)!\big(\chi_{i_1,j_1}\chi_{i_2,j_2} + \chi_{i_1,j_1}\chi_{\pi(j_3),\pi(i_3)} + \chi_{i_2,j_2}\chi_{\pi(j_3),\pi(i_3)}\end{displaymath}
\begin{displaymath}- \chi_{i_1,j_1}\chi_{i_3,j_3} - \chi_{i_1,j_1}\chi_{\pi(i_2),\pi(j_2)} - \chi_{i_3,j_3}\chi_{\pi(i_2),\pi(j_2)}- \chi_{i_2,j_2}\chi_{i_3,j_3} - \chi_{i_2,j_2}\chi_{\pi(i_1),\pi(j_1)} - \chi_{i_3,j_3}\chi_{\pi(i_1),\pi(j_1)}  \big).\end{displaymath}
\item[\begin{math}\triangleright\end{math}] If \begin{math}\pi(i_1) < \pi(j_1)\end{math}, \begin{math}\pi(i_2) > \pi(j_2)\end{math} and \begin{math}\pi(i_3) > \pi(j_3)\end{math}:
\begin{displaymath}[X_{i_1,j_1}X_{i_2,j_2}X_{i_3,j_3}]\mathtt{f_3}(\pi)= \frac{3n!^2}{4} + (n-2)!^2\end{displaymath}
\begin{displaymath}+ 2(n-4)(n-3)!^2\big(\chi_{i_1,j_1}\chi_{i_2,j_2}\chi_{\pi(j_3),\pi(i_3)} + \chi_{i_1,j_1}\chi_{\pi(j_2),\pi(i_2)}\chi_{i_3,j_3} - \chi_{\pi(i_1),\pi(j_1)}\chi_{i_2,j_2}\chi_{i_3,j_3}\big)\end{displaymath}
\begin{displaymath}+ 2(n-3)!(n-2)!\big(\chi_{i_1,j_1}\chi_{i_2,j_2} + \chi_{i_1,j_1}\chi_{\pi(j_3),\pi(i_3)} + \chi_{i_2,j_2}\chi_{\pi(j_3),\pi(i_3)}\end{displaymath}
\begin{displaymath}+ \chi_{i_1,j_1}\chi_{i_3,j_3} + \chi_{i_1,j_1}\chi_{\pi(j_2),\pi(i_2)} + \chi_{i_3,j_3}\chi_{\pi(j_2),\pi(i_2)}- \chi_{i_2,j_2}\chi_{i_3,j_3} - \chi_{i_2,j_2}\chi_{\pi(i_1),\pi(j_1)} - \chi_{i_3,j_3}\chi_{\pi(i_1),\pi(j_1)}  \big).\end{displaymath}
\item[\begin{math}\triangleright\end{math}] If \begin{math}\pi(i_1) > \pi(j_1)\end{math}, \begin{math}\pi(i_2) > \pi(j_2)\end{math} and \begin{math}\pi(i_3) > \pi(j_3)\end{math}:
\begin{displaymath}[X_{i_1,j_1}X_{i_2,j_2}X_{i_3,j_3}]\mathtt{f_3}(\pi)= \frac{3n!^2}{4} + 3(n-2)!^2\end{displaymath}
\begin{displaymath}+ 2(n-4)(n-3)!^2\big(\chi_{i_1,j_1}\chi_{i_2,j_2}\chi_{\pi(j_3),\pi(i_3)} + \chi_{i_1,j_1}\chi_{\pi(j_2),\pi(i_2)}\chi_{i_3,j_3} + \chi_{\pi(j_1),\pi(i_1)}\chi_{i_2,j_2}\chi_{i_3,j_3}\big)\end{displaymath}
\begin{displaymath}+ 2(n-3)!(n-2)!\big(\chi_{i_1,j_1}\chi_{i_2,j_2} + \chi_{i_1,j_1}\chi_{\pi(j_3),\pi(i_3)} + \chi_{i_2,j_2}\chi_{\pi(j_3),\pi(i_3)}\end{displaymath}
\begin{displaymath}+ \chi_{i_1,j_1}\chi_{i_3,j_3} + \chi_{i_1,j_1}\chi_{\pi(j_2),\pi(i_2)} + \chi_{i_3,j_3}\chi_{\pi(j_2),\pi(i_2)}+ \chi_{i_2,j_2}\chi_{i_3,j_3} + \chi_{i_2,j_2}\chi_{\pi(j_1),\pi(i_1)} + \chi_{i_3,j_3}\chi_{\pi(j_1),\pi(i_1)}  \big).\end{displaymath}
\end{itemize}
\end{itemize}
\end{invinvinv}

\begin{proof}
\begin{itemize}
\item[(a)] For \begin{math}(i_1,j_1)=(i_2,j_2)=(i_3,j_3)\end{math}: \begin{displaymath}[X_{i_1,j_1}^3]\mathtt{f_3}(\pi)= \sum_{a,b \in [n]} \# {}^{i_1\,j_1}\mathcal{S}_n \cap \mathcal{S}_n^{a\,b} \times \mathcal{S}_n^{\pi(i_1)\,\pi(j_1)}{}_{|\bar{i_1}=a,\bar{j_1}=b}.\end{displaymath}
\item[(b)] For \begin{math}(i_1,j_1)=(i_2,j_2) \neq (i_3,j_3)\end{math}:  \begin{displaymath}[X_{i_1,j_1}^2X_{i_3,j_3}]\mathtt{f_3}(\pi)= \sum_{a,b \in [n]} \# {}^{i_1\,j_1}\mathcal{S}_n \cap \mathcal{S}_n^{a\,b} \times \mathcal{S}_n^{\pi(i_3)\,\pi(j_3)}{}_{|\bar{i_1}=a,\bar{j_1}=b}\end{displaymath} \begin{displaymath}+ \sum_{a,b \in [n]} \# {}^{i_1\,j_1}\mathcal{S}_n \cap \mathcal{S}_n^{a\,b} \times \mathcal{S}_n^{\pi(i_1)\,\pi(j_1)}{}_{|\bar{i_3}=a,\bar{j_3}=b} + \sum_{a,b \in [n]} \# {}^{i_3\,j_3}\mathcal{S}_n \cap \mathcal{S}_n^{a\,b} \times \mathcal{S}_n^{\pi(i_1)\,\pi(j_1)}{}_{|\bar{i_1}=a,\bar{j_1}=b}.\end{displaymath}
\item[(c)] For \begin{math}(i_1,j_1) \neq (i_3,j_3) \neq (i_2,j_2)\end{math}:
\begin{displaymath}[X_{i_1,j_1} X_{i_2,j_2} X_{i_3,j_3}]\mathtt{f_3}(\pi)= \sum_{a,b \in [n]} \# {}^{i_1\,j_1}\mathcal{S}_n \cap \mathcal{S}_n^{a\,b} \times \mathcal{S}_n^{\pi(i_3)\,\pi(j_3)}{}_{|\bar{i_2}=a,\bar{j_2}=b}\end{displaymath} 
\begin{displaymath}+ \sum_{a,b \in [n]} \# {}^{i_2\,j_2}\mathcal{S}_n \cap \mathcal{S}_n^{a\,b} \times \mathcal{S}_n^{\pi(i_3)\,\pi(j_3)}{}_{|\bar{i_1}=a,\bar{j_1}=b}
+ \sum_{a,b \in [n]} \# {}^{i_1\,j_1}\mathcal{S}_n \cap \mathcal{S}_n^{a\,b} \times \mathcal{S}_n^{\pi(i_2)\,\pi(j_2)}{}_{|\bar{i_3}=a,\bar{j_3}=b}\end{displaymath} 
\begin{displaymath}+ \sum_{a,b \in [n]} \# {}^{i_3\,j_3}\mathcal{S}_n \cap \mathcal{S}_n^{a\,b} \times \mathcal{S}_n^{\pi(i_2)\,\pi(j_2)}{}_{|\bar{i_1}=a,\bar{j_1}=b}
+ \sum_{a,b \in [n]} \# {}^{i_2\,j_2}\mathcal{S}_n \cap \mathcal{S}_n^{a\,b} \times \mathcal{S}_n^{\pi(i_1)\,\pi(j_1)}{}_{|\bar{i_3}=a,\bar{j_3}=b}\end{displaymath}
\begin{displaymath}+ \sum_{a,b \in [n]} \# {}^{i_3\,j_3}\mathcal{S}_n \cap \mathcal{S}_n^{a\,b} \times \mathcal{S}_n^{\pi(i_1)\,\pi(j_1)}{}_{|\bar{i_2}=a,\bar{j_2}=b}.\end{displaymath} 
\end{itemize}
\end{proof}

\noindent Let \begin{math}n \geq 4\end{math}. We write
\begin{displaymath}\Lambda = (n-2)!\sum_{\{(i,j) \in [n]^2\ |\ i<j\}}(j-i)X_{i,j},\end{displaymath}
\begin{displaymath}\Delta = (n-3)!\sum_{\{(i,j) \in [n]^2\ |\ i<j\}} \Big(n-2(j-i)\Big) X_{i,j}.\end{displaymath}

\newtheorem{plus}[XMNFijI]{Lemma}
\begin{plus}
Let \begin{math}n \geq 4\end{math}. Then: 
\begin{itemize}
\item[(a)] \begin{displaymath}[X_{i,j}](\Lambda + \Delta)= 2(n-2)! + (n-4)(n-3)!\chi_{i,j}.\end{displaymath}
\item[(b)] \begin{displaymath}[X_{i_1,j_1}X_{i_2,j_2}](\Lambda \Delta) = (n-3)!(n-2)! \frac{\big( 2(n-2) + (n-4)\chi_{i_1,j_1} + (n-4)\chi_{i_2,j_2} - 4\chi_{i_1,j_1} \chi_{i_2,j_2} \big)}{1 + \delta_{ (i_1,j_1),(i_2,j_2)}}  \end{displaymath} 
\end{itemize}
\end{plus}

\begin{proof} Coefficient calculating.
\end{proof}

\noindent Let \begin{math}n \geq 4\end{math}. We define: 
\begin{displaymath}\left.\begin{array}{cccc}
\mathtt{f}: & \mathcal{S}_n & \rightarrow & \mathbb{R}[X_{1,2}, \dots, X_{n-1,n}]\\
   & \pi & \mapsto &  \Lambda \Delta \mathtt{inv_X}(\pi) + \big(\Lambda + \Delta \big) \mathtt{f_2}(\pi)  + \mathtt{f_3}(\pi)
\end{array}\right.\end{displaymath}

\noindent We write \begin{math}\iota\end{math} for the identity permutation of \begin{math}\mathcal{S}_n\end{math}.

\newtheorem{richtiges}[XMNFijI]{Lemma}
\begin{richtiges}
Let \begin{math}n \geq 4\end{math} and \begin{math}\pi \in \mathcal{S}_n\end{math}. Then:
\begin{displaymath}\mathtt{f}(\pi)= \mathtt{f}(\iota)= \big(\Lambda + \Delta \big) \mathtt{f_2}(\iota)  + \mathtt{f_3}(\iota).\end{displaymath}
\end{richtiges}

\begin{proof} Using Lemma 4.3, Lemma 4.5 and Lemma 4.6, we prove that:
\begin{itemize}
\item for \begin{math}(i_1,j_1) = (i_2,j_2) = (i_3,j_3)\end{math}:
\begin{displaymath}[X_{i_1, j_1}^3]\mathtt{f}(\iota) = [X_{i_1, j_1}^3]\mathtt{f}(\pi)\end{displaymath}
\begin{displaymath}= \frac{n!^2}{8} + (n-2)!^2\frac{n^2-n+1}{2} + (n-4)(n-3)!(n-2)!\frac{n^2-n+2}{4}\chi_{i_1,j_1} - (n-3)!(n-2)!\chi_{i_1,j_1}^2,\end{displaymath}
\item for \begin{math}(i_1,j_1) = (i_2,j_2) \neq (i_3,j_3)\end{math}:
\begin{displaymath}[X_{i_1, j_1}^2X_{i_3, j_3}]\mathtt{f}(\iota) = [X_{i_1, j_1}^2X_{i_3, j_3}]\mathtt{f}(\pi)  \end{displaymath}
\begin{displaymath}= \frac{3n!^2}{8}+ \frac{3(n-2)!n!}{2} + \frac{3(n-2)!^2}{2}\end{displaymath} \begin{displaymath}+ (n-4)(n-3)!(n-2)!\frac{n^2-n+2}{4}\big(2\chi_{i_1,j_1} + \chi_{i_3,j_3}\big)\end{displaymath} \begin{displaymath}- (n-3)!(n-2)!\big(\chi_{i_1,j_1}^2 + 2\chi_{i_1,j_1}\chi_{i_3,j_3}\big),\end{displaymath}
\item for \begin{math}(i_1,j_1) \neq (i_2,j_2) \neq (i_3,j_3) \neq (i_1,j_1)\end{math}:
\begin{displaymath}[ X_{i_1, j_1}X_{i_2, j_2}X_{i_3, j_3}]\mathtt{f}(\iota) = [ X_{i_1, j_1}X_{i_2, j_2}X_{i_3, j_3}]\mathtt{f}(\pi)  \end{displaymath}
\begin{displaymath}= \frac{3n!^2}{4} + 3(n-2)!n! + 3(n-2)!^2\end{displaymath}
\begin{displaymath}+ (n-4)(n-3)!(n-2)!\frac{n^2-n+2}{2}\big(\chi_{i_1,j_1} + \chi_{i_2,j_2} + \chi_{i_3,j_3}\big)\end{displaymath}
\begin{displaymath}- 2(n-3)!(n-2)!\big(\chi_{i_1,j_1}\chi_{i_2,j_2} + \chi_{i_1,j_1}\chi_{i_3,j_3} + \chi_{i_2,j_2}\chi_{i_3,j_3}\big).\end{displaymath}
\end{itemize}
\end{proof}

\noindent We are now in position to determine the minimal polynomial of \begin{math}\mathfrak{I_n}\end{math}.

\noindent Let \begin{math}n \geq 4\end{math}. We write
\begin{displaymath}\Omega = \frac{n!}{2}\sum_{\{(i,j) \in [n]^2\ |\ i<j\}}X_{i,j}.\end{displaymath}

\newtheorem{minpoly}[XMNFijI]{Proposition}
\begin{minpoly}
Let \begin{math}n \geq 4\end{math}. The minimal polynomial of \begin{math}\mathfrak{I_n}\end{math} is:
\begin{displaymath}X \big( X + \Lambda \big) \big( X + \Delta \big) \big( X - \Omega \big).\end{displaymath}
\end{minpoly}

\begin{proof} Using Lemma 2.4 and Lemma 4.7, we get:
\begin{displaymath}\big( \sum_{\sigma \in \mathcal{S}_n} \mathtt{inv_X}(\sigma)\sigma + \Lambda \iota \big) \big( \sum_{\sigma \in \mathcal{S}_n} \mathtt{inv_X}(\sigma)\sigma + \Delta \iota \big) \sum_{\sigma \in \mathcal{S}_n} \mathtt{inv_X}(\sigma)\sigma\end{displaymath}
\begin{displaymath}= \sum_{\sigma \in \mathcal{S}_n} \big( \sum_{\begin{subarray}{l} \sigma_1, \sigma_2, \sigma_3 \in \mathcal{S}_n \\ \sigma_1 \sigma_2 \sigma_3 = \sigma \end{subarray}} \mathtt{inv_X}(\sigma_1)\mathtt{inv_X}(\sigma_2)\mathtt{inv_X}(\sigma_3) \big)\sigma + \sum_{\sigma \in \mathcal{S}_n} \big( (\Lambda + \Delta)\sum_{\begin{subarray}{l} \sigma_1, \sigma_2 \in \mathcal{S}_n \\ \sigma_1 \sigma_2 = \sigma \end{subarray}} \mathtt{inv_X}(\sigma_1)\mathtt{inv_X}(\sigma_2) \big)\sigma + \sum_{\sigma \in \mathcal{S}_n \backslash \{\iota\}} \Lambda \Delta \mathtt{inv_X}(\sigma) \sigma\end{displaymath}
\begin{displaymath}= \sum_{\sigma \in \mathcal{S}_n} \big( \sum_{\sigma_1, \sigma_2 \in \mathcal{S}_n} \mathtt{inv_X}(\sigma_1^{-1})\mathtt{inv_X}(\sigma_1 \sigma_2^{-1})\mathtt{inv_X}(\sigma_2 \sigma) \big)\sigma + \sum_{\sigma \in \mathcal{S}_n} \big( (\Lambda + \Delta)\sum_{\sigma_1 \in \mathcal{S}_n} \mathtt{inv_X}(\sigma_1^{-1})\mathtt{inv_X}(\sigma_1 \sigma) \big)\sigma + \sum_{\sigma \in \mathcal{S}_n \backslash \{\iota\}} \Lambda \Delta \mathtt{inv_X}(\sigma) \sigma\end{displaymath}
\begin{displaymath}= \sum_{\sigma \in \mathcal{S}_n} \mathtt{f}(\sigma) \sigma = \mathtt{f}(\iota) \sum_{\sigma \in \mathcal{S}_n}\sigma.\end{displaymath}
Then:
\begin{displaymath}\big( \sum_{\sigma \in \mathcal{S}_n} \mathtt{inv_X}(\sigma)\sigma + \Lambda \iota \big) \big( \sum_{\sigma \in \mathcal{S}_n} \mathtt{inv_X}(\sigma)\sigma + \Delta \iota \big) \sum_{\sigma \in \mathcal{S}_n} \mathtt{inv_X}(\sigma)\sigma \big( \sum_{\sigma \in \mathcal{S}_n} \mathtt{inv_X}(\sigma)\sigma - \Omega \iota \big)\end{displaymath}
\begin{displaymath}= \mathtt{f}(\iota) \sum_{\sigma \in \mathcal{S}_n}\sigma \big( \sum_{\sigma \in \mathcal{S}_n} \mathtt{inv_X}(\sigma)\sigma - \Omega \iota \big) = 0. \end{displaymath}
Hence the minimal polynomial of \begin{math}\mathfrak{I_n}\end{math} divides \begin{math}X \big( X + \Lambda \big) \big( X + \Delta \big) \big( X - \Omega \big)\end{math}.\\
It is clear that the minimal polynomial of \begin{math}\mathfrak{I_n}\end{math} does not divide \begin{math}X \big( X + \Lambda \big) \big( X + \Delta \big)\end{math}.\\
We have:
\begin{displaymath}[ X_{1, 4}^3]\bigg( [\iota]\Big( \big( \sum_{\sigma \in \mathcal{S}_n} \mathtt{inv_X}(\sigma)\sigma + \Lambda \iota \big) \big( \sum_{\sigma \in \mathcal{S}_n} \mathtt{inv_X}(\sigma)\sigma - \Omega \iota \big) \sum_{\sigma \in \mathcal{S}_n} \mathtt{inv_X}(\sigma)\sigma  \Big) \bigg)\end{displaymath}
\begin{displaymath}= (n-3)!^2\frac{n^4 - 8n^3 + 22n^2 - 36n + 44}{2} \neq 0.\end{displaymath}
Then the  minimal polynomial of \begin{math}\mathfrak{I_n}\end{math} does not divide \begin{math}X \big( X + \Lambda \big) \big( X - \Omega \big)\end{math}.\\
We have:
\begin{displaymath}[X_{1, 3}^3]\bigg( [\iota]\Big( \big( \sum_{\sigma \in \mathcal{S}_n} \mathtt{inv_X}(\sigma)\sigma + \Delta \iota \big) \big( \sum_{\sigma \in \mathcal{S}_n} \mathtt{inv_X}(\sigma)\sigma - \Omega \iota \big) \sum_{\sigma \in \mathcal{S}_n} \mathtt{inv_X}(\sigma)\sigma  \Big) \bigg)\end{displaymath}
\begin{displaymath}= -4(n-3)!(n-2)! - (n-3)!n! \neq 0.\end{displaymath}
Then the  minimal polynomial of \begin{math}\mathfrak{I_n}\end{math} does not divide \begin{math}X \big( X + \Delta \big) \big( X - \Omega \big)\end{math}.\\
We have:
\begin{displaymath}[X_{1, 3}^3]\bigg( [\iota]\Big( \big( \sum_{\sigma \in \mathcal{S}_n} \mathtt{inv_X}(\sigma)\sigma + \Lambda \iota \big) \big( \sum_{\sigma \in \mathcal{S}_n} \mathtt{inv_X}(\sigma)\sigma + \Delta \iota \big) \big( \sum_{\sigma \in \mathcal{S}_n} \mathtt{inv_X}(\sigma)\sigma - \Omega \iota \big)  \Big)  \bigg)\end{displaymath}
\begin{displaymath}= \frac{(n-2)!^2}{2} + \frac{(n-2)!n!}{2} - \frac{(n-2)!^2 n!}{2} \neq 0.\end{displaymath}
Then the  minimal polynomial of \begin{math}\mathfrak{I_n}\end{math} does not divide \begin{math}\big( X + \Lambda \big) \big( X + \Delta \big) \big( X - \Omega \big)\end{math}. \end{proof}

\section{Proofs of the Theorems}
\label{sec:proofs}

\noindent We prove first Theorem 1 and then Theorem 2.

\begin{proof} From Proposition 3.3, we deduce that \begin{math}\mathfrak{D_n}\end{math} is diagonalizable and: \begin{displaymath}Sp(\mathfrak{D_n})=\{\frac{n!}{2}\mathfrak{d_n},\ -(n-2)!\mathfrak{d_n},\ 0\}.\end{displaymath}
From the diagonalisability of \begin{math}\mathfrak{D_n}\end{math} and the result \eqref{Dn} of Section 2, we get \begin{displaymath}V_{\mathfrak{D_n}}\big(\frac{n!}{2}\mathfrak{d_n}\big)=1\end{displaymath}.\\
The trace of \begin{math}\mathfrak{D_n}\end{math} is \begin{math}0\end{math}. Then:
\begin{displaymath}\frac{n!}{2}\mathfrak{d_n}\, V_{\mathfrak{D_n}}\big(\frac{n!}{2}\mathfrak{d_n}\big)- (n-2)!\mathfrak{d_n}\, V_{\mathfrak{D_n}}\big(-(n-2)!\mathfrak{d_n}\big)+ 0\,V_{\mathfrak{D_n}}(0)= 0\end{displaymath}
\begin{displaymath}V_{\mathfrak{D_n}}\big(-(n-2)!\mathfrak{d_n}\big)= \binom{n}{2}.\end{displaymath}
The dimension of \begin{math}\mathfrak{D_n}\end{math} is \begin{math}n!\end{math}. Then:
\begin{displaymath}V_{\mathfrak{D_n}}\big(\frac{n!}{2}\mathfrak{d_n}\big)+ V_{\mathfrak{D_n}}\big(-(n-2)!\mathfrak{d_n}\big)+ V_{\mathfrak{D_n}}(0)= n!\end{displaymath}
\begin{displaymath}V_{\mathfrak{D_n}}(0)= n! -\binom{n}{2} -1.\end{displaymath}
\end{proof}

\begin{proof} From Proposition 4.8, we deduce that \begin{math}\mathfrak{I_n}\end{math} is diagonalizable and: \begin{displaymath}Sp(\mathfrak{I_n})=\{\Omega, -\Lambda, -\Delta, 0\}.\end{displaymath}
From the diagonalisability of \begin{math}\mathfrak{I_n}\end{math} and the result \eqref{In} of Section 2, we get \begin{displaymath}V_{\mathfrak{I_n}}(\Omega)=1\end{displaymath}.\\
The trace of \begin{math}\mathfrak{I_n}\end{math} is \begin{math}0\end{math}. Then:
\begin{itemize}
\item \begin{displaymath}V_{\mathfrak{I_n}}(\Omega) \big(\Omega \,,\, X_{1,2}\big) + V_{\mathfrak{I_n}}(\Lambda) \big(\Lambda \,,\, X_{1,2}\big) + V_{\mathfrak{I_n}}(\Delta) \big(\Delta \,,\, X_{1,2}\big) + V_{\mathfrak{I_n}}(0) \big(0 \,,\, X_{1,2}\big) = 0,\end{displaymath}
\begin{displaymath}V_{\mathfrak{I_n}}(\Lambda) + V_{\mathfrak{I_n}}(\Delta) = \binom{n}{2}.\end{displaymath}
\item \begin{displaymath}V_{\mathfrak{I_n}}(\Omega) \big(\Omega \,,\, X_{1,n}\big) + V_{\mathfrak{I_n}}(\Lambda) \big(\Lambda \,,\, X_{1,n}\big) + V_{\mathfrak{I_n}}(\Delta) \big(\Delta \,,\, X_{1,n}\big) + V_{\mathfrak{I_n}}(0) \big(0 \,,\, X_{1,n}\big) = 0,\end{displaymath}
\begin{displaymath}(n-1)V_{\mathfrak{I_n}}(\Lambda) - V_{\mathfrak{I_n}}(\Delta) = \binom{n}{2}.\end{displaymath}
\end{itemize}
So we deduce:
\begin{displaymath}V_{\mathfrak{I_n}}(\Lambda)=n-1\ \text{and}\ V_{\mathfrak{I_n}}(\Delta) = \binom{n-1}{2}.\end{displaymath}
The dimension of \begin{math}\mathfrak{I_n}\end{math} is \begin{math}n!\end{math}. Then:
\begin{displaymath}V_{\mathfrak{I_n}}(\Omega) + V_{\mathfrak{I_n}}(\Lambda) + V_{\mathfrak{I_n}}(\Delta) + V_{\mathfrak{I_n}}(0) = n!,\end{displaymath}
\begin{displaymath}V_{\mathfrak{I_n}}(0) = n! - \binom{n}{2} - 1.\end{displaymath}
\end{proof}

\bibliographystyle{abbrvnat}

\label{sec:biblio}

\end{document}